\documentclass[12pt,letterpaper]{article}
\pdfoutput=1
\usepackage[letterpaper,margin=25mm,marginparwidth=23mm]{geometry}
\usepackage{amsmath,amsthm,amsopn,amssymb}
\usepackage{mathtools}
\mathtoolsset{showonlyrefs}
\usepackage{tikz}
\usepackage{tikz-cd}
\usetikzlibrary{matrix}
\usetikzlibrary{fit}
\usetikzlibrary{decorations.pathreplacing}
\tikzset{
  symbol/.style={
    draw=none,
    every to/.append style={
      edge node={node [sloped, allow upside down, auto=false]{$#1$}}}
  }
}
\tikzset{
  LA/.style = {line width=#1, -{Straight Barb[length=3pt]}},
         LA/.default=1pt
}
\usepackage{thmtools}
\usepackage{xspace}
\usepackage{doi}

\usepackage{eucal}

\usepackage{color}
\definecolor{verydarkblue}{rgb}{0,0,0.4}
\usepackage{hyperref}
\hypersetup{
pdfauthor={David Dumas and Andrew Sanders},
pdftitle={Uniformization of compact complex manifolds by Anosov representations},
colorlinks=true,linkcolor=verydarkblue,
citecolor=verydarkblue,
urlcolor=verydarkblue
}

\usepackage{caption}
\newcommand{\noproof}{\hfill\qedsymbol}
\renewcommand{\bar}{\overline}

\newcommand{\cv}{Chevalley-Bruhat\xspace}

\newcommand{\into}{\hookrightarrow}

\newcommand{\dbar}{\bar{\partial}}

\newcommand{\R}{\mathbb{R}}
\newcommand{\C}{\mathbb{C}}

\newcommand{\B}{\mathrm{B}}
\newcommand{\F}{\mathcal{F}}
\newcommand{\A}{\mathcal{A}}
\newcommand{\uA}{\underline{\A}}
\newcommand{\W}{\mathcal{W}}
\newcommand{\KS}{\mathcal{KS}}

\newcommand{\lv}{\prescript{v\!}{}{}}
\newcommand{\lh}{\prescript{h\!}{}{}}

\renewcommand{\L}{\mathcal{L}}

\DeclareMathOperator{\coker}{coker}

\DeclareMathOperator{\Aut}{Aut}
\DeclareMathOperator{\hdim}{Hdim}
\DeclareMathOperator{\Diff}{Diff}
\DeclareMathOperator{\Mod}{Mod}

\renewcommand{\O}{\mathcal{O}}
\newcommand{\CP}{\mathbb{P}_\C}

\newcommand{\T}{\mathcal{T}}
\newcommand{\X}{\mathcal{X}}
\newcommand{\Y}{\mathcal{Y}}
\newcommand{\uY}{\underline{\mathcal{Y}}}

\renewcommand{\tilde}[1]{\widetilde{#1}}
\renewcommand{\Tilde}{\tilde}

\renewcommand{\sl}{\mathfrak{sl}}

\renewcommand{\phi}{\varphi}

\makeatletter
\newcommand*{\shifttext}[2]{%
  \settowidth{\@tempdima}{#2}%
  \makebox[\@tempdima]{\hspace*{#1}#2}%
}
\makeatother

\newsavebox{\foobox}

\newcommand{\g}{\mathfrak{g}}
\newcommand{\z}{\mathfrak{z}}

\renewcommand{\leq}{\leqslant}
\renewcommand{\geq}{\geqslant}

\renewcommand{\rho}{\varrho}

\newcommand{\SO}{\mathrm{SO}}
\newcommand{\PSL}{\mathrm{PSL}}

\newcommand{\Hom}{\mathrm{Hom}}

\newcommand{\Lie}{\mathrm{Lie}}
\newcommand{\Ad}{\mathrm{Ad}}

\renewcommand{\hat}{\widehat}

\newcommand{\colorboxed}[1]{%
  \begingroup
    \colorlet{cb@saved}{.}%
    \color{red}%
    \boxed{%
      \color{cb@saved}%
      #1%
    }%
  \endgroup
}

\newcommand{\param}{*}

\numberwithin{equation}{section}
\theoremstyle{plain}
\newtheorem{thm}{Theorem}[section]
\newtheorem{cor}[thm]{Corollary}

\newtheorem{prop}[thm]{Proposition}
\newtheorem{defn}[thm]{Definition}

\theoremstyle{definition}

\newtheorem{example}{Example}

\theoremstyle{definition}
\newtheorem*{remark}{Remark}
\newtheorem*{remarksenv}{Remarks}

\newenvironment{rmenumerate}{\begin{enumerate}}{\end{enumerate}}

\begin{document}

\title{Uniformization of compact complex manifolds by Anosov homomorphisms}

\author{David Dumas and Andrew Sanders}

\date{June 4, 2021\footnote{Release history: v1: October 10, 2020.  v2 (this version): June 4, 2021.}}

\maketitle

\section{Introduction}

In this paper we study complex structures on compact manifolds that arise as quotients of domains in complex flag varieties by Anosov subgroups.
The central question we consider is one of \emph{uniformization}, that is, determining which complex structures on these smooth manifolds arise as a quotient of such a domain in the flag variety.  We obtain infinitesimal and local uniformization results in cases where the Anosov subgroups have limit sets of sufficiently small Riemannian Hausdorff dimension.

The prototype for this program is the Bers Simultaneous Uniformization Theorem \cite{BER60}---a global result with no Hausdorff dimension hypothesis---and so we begin by recalling its statement.

Let $S$ be a closed, connected, oriented $2$-dimensional manifold, and let $\Gamma = \pi_1(S)$.
A discrete injective homomorphism $\rho: \Gamma\rightarrow \PSL(2, \C)$ is \emph{quasi-Fuchsian} if there exists a $\rho$-equivariant continuous injective map $\xi_{\rho}: \partial \Gamma\rightarrow \CP^{1}$.
The complement of the Jordan curve $\xi_{\rho}(\partial \Gamma)$ is a pair of $\rho(\Gamma)$-invariant topological disks, which yields a pair of quotient Riemann surfaces $Y_\rho^\pm$.
The pair $(Y_\rho^+, Y_\rho^-)$ determines a point in the Teichm\"uller space $\T(S \cup \bar{S}) = \T(S) \times \T(\bar{S})$, where $\bar{S}$ denotes $S$ with the opposite orientation.

The set of quasi-Fuchsian homomorphisms $\mathcal{QF}_{S}\subset \Hom(\Gamma, \PSL(2, \C))$ is a non-empty smooth subset that is invariant under the action of $\PSL(2, \C)$ by conjugation.
The action of $\PSL(2, \C)$ on $\mathcal{QF}_{S}$ is free and proper, giving a quotient complex manifold $\underline{\mathcal{QF}}_{S}$, and the map $\rho \mapsto (Y_\rho^+,Y_\rho^-)$ descends to give a \emph{classifying map} $\underline{\mathcal{QF}}_{S} \to \T(S \cup \bar{S})$.

Bers' Simultaneous Uniformization Theorem asserts that this classifying map is a biholomorphism.
Equivalently, working at the level of homomorphisms, the map $\mathcal{QF}_{S} \to \mathcal{T}(S \cup \bar{S})$, $\rho \mapsto (Y_\rho^+,Y_\rho^-)$ is a surjective submersion whose fibers are $\PSL(2,\C)$-orbits.
It follows that any deformation of complex structure on $(Y_\rho^+,Y_\rho^-)$ can be lifted to a deformation of quasi-Fuchsian homomorphisms.

Replacing $\PSL(2,\C) = \Aut(\CP^1)$ with the automorphism group $G$ of a higher dimensional complex flag variety, the set of \emph{Anosov homomorphisms} provides a natural generalization of quasi-Fuchsian homomorphisms (see e.g. \cite{LAB06}\cite{GW12}\cite{KLP13}).
Precisely, given a complex semisimple Lie group $G$, a symmetric parabolic subgroup $P_A<G,$ and a torsion-free word hyperbolic group $\Gamma,$ there exists a (potentially empty) open subset $\A\subset \Hom(\Gamma, G)$ of \emph{$P_A$-Anosov homomorphisms}.
Every $\rho\in \A$ is discrete, injective, and admits a unique $\rho$-equivariant continuous injective \emph{limit map} $\xi_{\rho}: \partial \Gamma\rightarrow G/P_A.$

Now, denoting by $\F=G/P_{D}$ a complex flag variety of $G$, where $P_{D}<G$ is a parabolic subgroup,
let $\rho: \Gamma\rightarrow G$ be a $P_{A}$-Anosov homomorphism.
Given a subset $I$ of the Weyl group $W = W(G)$ satisfying certain conditions (elaborated on in \autoref{subsec:domains}), the work of Kapovich-Leeb-Porti \cite{KLP13} yields a domain $\Omega_{\rho}^{I}\subset \F$ on which $\rho(\Gamma)$ acts freely, properly discontinuously, and cocompactly.  The subset $I\subset W$ satisfying these properties is called a \emph{balanced ideal of type $(P_{A}, \F)$}.

Allowing $\rho$ to vary, and restricting attention to the set $\A_I$ of representations where the domain $\Omega_\rho^I$ is nonempty, the quotients of these domains give a family of compact complex manifolds $\W^{I}\rightarrow \A_I$, which we call the \emph{Anosov family}.
When $G=\PSL(2, \C)$ and $\Gamma=\pi_{1}(S)$ is a surface group, then $\A=\mathcal{QF}_{S}$ and there is a unique choice for $I;$ the resulting family consists of the pairs $(Y_\rho^+,Y_\rho^-)$ for $\rho \in \mathcal{QF}_{S}$.

For $G$ of higher rank, there are many possible choices for $\F$ and $I$, and hence many (a priori) distinct Anosov families that can be seen as generalizations of the quasi-Fuchsian case.
In light of Bers' theorem, it is natural to ask whether these families also realize all possible deformations of complex structure of their fibers, either infinitesimally, locally, or globally.

Given a $P_{A}$-Anosov homomorphism $\rho: \Gamma\rightarrow G$ and $I \subset W$ a balanced ideal of type $(P_{A}, \F),$ we call the pair $(\rho, I)$ a \emph{thickened Anosov homomorphism}.  
If the associated domain $\Omega_{\rho}^{I}\subset \F$ is the complement of a closed set with Riemannian Hausdorff codimension at least $4$,
we say that the thickened Anosov homomorphism $(\rho,I)$ is \emph{$4$-small}; see \autoref{sec:prelim} for a detailed definition.

In this setting, we first establish an infinitesimal analogue of Bers' theorem.
This is most naturally phrased in terms of surjectivity of the Kodaira-Spencer map, which measures first-order changes in the complex structure of $\W_{\rho}^{I}$ via the first cohomology group of the tangent sheaf $\Theta_{\W_\rho^I}$.

\begin{thm}
\label{thm:main-infinitesimal}
Let $\W^I \to \A_I$ denote the Anosov family of compact complex manifolds associated to ($\mathcal{F}$, $G$, $P_A$, $I$) as above, and suppose $G = \Aut(\F)$.
If $(\rho,I)$ is $4$-small, then:

\begin{rmenumerate}
\item The Anosov family realizes all first-order deformations: The Kodaira-Spencer map $\KS_{\rho}: T_{\rho}\A\rightarrow H^{1}(\W_{\rho}^{I}, \Theta_{\W_{\rho}^{I}})$ is surjective.

\item Infinitesimal $G$-conjugations are the only trivial first-order deformations in the Anosov family: Under the natural isomorphism $T_\rho\A \simeq Z^1(\Gamma,\g_\rho)$, where $\g_\rho$ denotes $\g=\Lie(G)$ with the $\Gamma$-module structure from $\Ad \circ \rho$, the kernel of $\KS_\rho$ is isomorphic to $B^{1}(\Gamma, \g_{\rho})$.

\item All infinitesimal automorphisms of $\W_\rho^I$ arise from the infinitesimal centralizer of $\rho$:  There is an isomorphism $H^{0}(\W_{\rho}^{I}, \Theta_{\W_{\rho}^{I}}) \simeq H^{0}(\Gamma, \g_{\rho}) = \Lie(Z_G(\rho(\Gamma))$.
\end{rmenumerate}
\end{thm}

In essence, \autoref{thm:main-infinitesimal} asserts that when $(\rho, I)$ is $4$-small, the first order deformation theory of the homomorphism $\rho$ is equivalent to the first order deformation theory of the associated complex manifold $\W_{\rho}^{I}$; complex-analytic deformation questions thus correspond to representation-theoretic ones.

In particular this gives a characterization of Anosov representations with complex-analytically rigid quotient manifolds:
\begin{cor}
\label{cor:main-infinitesimal}
Let $(\rho, I)$ be a thickened Anosov homomorphism of type $(P_{A}, \F)$ such that $(\rho, I)$ is $4$-small. Then the complex manifold $\W_{\rho}^{I}$ is infinitesimally rigid if and only if the  homomorphism $\rho: \Gamma\rightarrow G$ is infinitesimally rigid modulo conjugation.
\end{cor}

Using \autoref{cor:main-infinitesimal} in conjunction with some rigidity results in the theory of discrete groups, in \autoref{ex:rigid} we produce new examples of rigid complex manifolds.  
To place this result in context, we note that it was a long standing conjecture of Kodaira that if a compact K\"{a}hler manifold $X$ is rigid, then $X$ is projective.  This was eventually disproved by Voisin \cite{VOI06}.
In \cite{DS20}, the authors showed that the manifolds studied here are \emph{not} K\"{a}hler, hence our examples are of a complementary nature.

Next we move from infinitesimal to local statements, working with the generalized Teichm\"uller spaces of the complex manifolds $\W_{\rho}^{I}$.  Recall that if $Y$ is a closed, oriented smooth manifold, the Teichm\"{u}ller space $\mathcal{T}(Y)$ of $Y$ is the space of all complex 
structures on $Y$ compatible with its orientation, modulo the group of diffeomorphisms isotopic to the identity.
In contrast to the Riemann surface case, the Teichm\"uller spaces of higher-dimensional manifolds are often pathological (e.g.~not locally Hausdorff).
Given a family of compact complex manifolds $\Y\rightarrow B$ whose (smooth) structure group is reduced to $\textnormal{Diff}_{0}$ and any $b\in B,$ there is a natural \emph{classifying map} $B\rightarrow \mathcal{T}(\Y_{b})$ which records the marked complex structures of the fibers.

In studying this map for the Anosov family, it is natural to consider the quotient of the set $\A$ of Anosov representations by the conjugation action of $G$.
This introduces significant complications, as the action is in general neither free nor proper.
Nevertheless, we find reasonable conditions (including the $4$-small condition) under which the quotient of $\A$ is locally isomorphic to Teichm\"uller space.
When $\Gamma$ is a surface group, we obtain the following result:

\begin{thm}
\label{thm:main-local}
Let $S$ be a closed orientable surface and $\Gamma=\pi_{1}S.$ Let $G$ be a complex simple adjoint Lie group not of type $A_{1}, A_{2}, A_{3}$ or $B_{2},$ and $I$ any balanced ideal of type $(P_{A}, \F)$ such that that $G$-quasi-Fuchsian and $G$-Hitchin representations give rise to nonempty domains $\Omega_\rho^I$.  Then there exists a non-empty connected open $G$-invariant manifold $\mathcal{U}\subset \A\subset \Hom(\Gamma, G)$ such that
\begin{rmenumerate}
\item  The set $\mathcal{U}$ contains all $G$-quasi-Fuchsian homomorphisms.  If we further omit types $F_{4}, E_{6}, E_{7}, E_{8},$ then $\mathcal{U}$ also contains all $G$-Hitchin homomorphisms.
\item The $G$-action on $\mathcal{U}$ is free and proper, with quotient manifold $\underline{\mathcal{U}}\subset \uA$ that is $1$-connected.  
\item For any $\rho\in \mathcal{U},$ the classifying map $f: \underline{\mathcal{U}}\rightarrow \T(\W_{\rho}^{I})$ of the family $\W^I$
is open, locally surjective, and is a homeomorphism onto its image.
\item There is a commutative diagram of holomorphic maps
\begin{center}
\begin{tikzcd}
\underline{\mathcal{QF}}_{S} \arrow{r} \arrow{d}
&\T(S) \times \T(\overline{S}) \arrow{d} \\
\underline{\mathcal{U}} \arrow{r} 
& \T(\W_{\rho}^{I}).
\end{tikzcd}
\end{center}
where the top horizontal arrow is the simultaneous uniformization map.
\end{rmenumerate}
\end{thm}

The notions of $G$-quasi-Fuchsian and $G$-Hitchin homomorphisms are recalled in \autoref{sec: flag}.

We emphasize that \autoref{thm:main-local}(iii) shows that Anosov homomorphisms satisfying these hypotheses realize all small deformations of the complex structure on $\W_{\rho}^{I}$, and with (iv) this shows that the Bers simultaneous uniformization isomorphism admits a local analytic continuation to the setting of higher-rank complex Anosov homomorphisms.

The proof of \autoref{thm:main-local} hinges on the recent work of Pozzetti-Sambarino-Wienhard \cite{PSW19} which implies that the condition of $4$-smallness persists throughout an open set in $\A$.

For general $\Gamma$ (not a surface group), an analogue of \autoref{thm:main-local} is proved in \autoref{thm:uniformization}, but the assumptions become significantly more technical since we must require certain additional hypotheses that are automatic for surface groups.

\subsection{Extensions and complements}

Due to the presence of singularities in $\Hom(\Gamma, G),$ our discussion of spaces of Anosov homomorphisms uses the language of complex analytic spaces.  However, to obtain a theory that fully incorporates the $G$-action on $\Hom(\Gamma, G)$, it may be necessary to use the even more general framework of complex analytic stacks.  This is made especially clear by the existence of non-reductive Anosov homomorphisms (see e.g.~\cite{GGKW15}).

Regardless of the formalism, the results in the paper show that the naive generalization of Bers' Simultaneous Uniformization Theorem to Anosov homomorphisms will be false in general. Indeed, we exhibit in \autoref{ex:GHYS} a case where the Anosov homomorphism is rigid, but the corresponding complex manifold can be deformed.
But in some cases, for instance surface groups, there is still hope for a general statement. 
We will formulate one question along these lines which has some hope of being true.

Let $\Gamma=\pi_{1}(S)$ be a closed surface group and $G$ be a complex simple Lie group of adjoint type.  Let $\A\subset \Hom(\Gamma, G)$ denote the connected component of $P_{A}$-Anosov homomorphisms containing the $G$-Fuchsian homomorphisms.  Then there is a well-defined quotient stack $[\A/G]$.

Now, let $I$ be a balanced ideal of type $(P_{A}, \F)$ and $\W^{I}\rightarrow \A$ the corresponding family.  Let $\rho \in \A$ and consider the \emph{Teichm\"{u}ller stack} $\mathcal{T}(\W_{\rho}^{I})$ as recently defined by Meersseman \cite{MEE19}.
Is there an isomorphism  $[\A/G]\simeq \mathcal{T}(\W_{\rho}^{I})$ of complex analytic stacks extending Bers' Simultaneous Uniformization Theorem?  

An affirmative answer to this question would be a first step to understanding non-Anosov limits of Anosov homomorphisms via complex analysis, in the spirit of results in the theory of Kleinian groups such as Thurston's double limit theorem and the ending lamination theorem.  

Finally, the prominent role of homological algebra in this paper can be seen as a substitute for the special role of the Beltrami equation and the Measurable Riemann Mapping Theorem in the theory of Kleinian groups in $\PSL(2, \C)$; a similar phenomenon is seen in the theory of higher-dimensional Kleinian groups, as discussed by Kapovich in \cite{KAP08}.
To wit, the Bers Simultaneous Uniformization Theorem is a straightforward consequence of the Measurable Riemann Mapping Theorem, and the lack of a satisfactory theory of quasi-conformal maps for higher dimensional flag varieties forces one to turn to the techniques of homological algebra in attacking analogous questions.

\subsection{Outline}
\autoref{sec:prelim} collects preliminaries on Lie theory, flag varieties, Anosov homomorphisms, and domains of discontinuity.  For readers familiar with the theory of Anosov representations and their actions on flag manifolds, this section serves only to set notation and conventions.

\autoref{sec:anosov} discusses complex analytic families and deformation theory, and defines the families of complex manifolds over the set of Anosov homomorphisms which are the main objects of study in the paper.

\autoref{sec:ks} contains a key technical result (\autoref{thm:pent}) that identifies the Kodaira-Spencer map of the families introduced in \autoref{sec:anosov} with an edge map in a certain spectral sequence.  The proofs of the main results use this theorem.

\autoref{sec:small} contains the general theorems which allow a comparison between the deformation theory of Anosov homomorphisms and the deformation theory of complex manifolds appearing in the associated families, including the proofs of \autoref{thm:main-infinitesimal} and \autoref{cor:main-infinitesimal}.  It is in this section where the hypothesis of being $4$-small plays a critical role.

\autoref{sec:teich} and \autoref{sec:small-teich} introduce Teichm\"{u}ller spaces and use the results of \autoref{sec:small} to deduce the promised local uniformization results (including \autoref{thm:main-local}) extending the Bers Simultaneous Uniformization Theorem.

Finally, \autoref{sec:son1} contains a brief discussion of cocompact lattices in $\SO(n,1)$ in the context of this paper.  Here, we highlight the relationship between deformations of convex real projective manifolds and deformations of the associated complex manifolds.

\subsection{Acknowledgements}

The authors thank Beatrice Pozzetti and Anna Wienhard for helpful conversations related to this work, and an anonymous referee for helpful suggestions and corrections.
The first author was supported by the U.S.~National Science Foundation, through award DMS 1709877.  The second author was supported by the Deutsche Forschungsgemeinschaft within the RTG 2229 “Asymptotic
invariants and limits of groups and spaces" and by Deutsche Forschungsgemeinschaft (DFG, German Research Foundation) under Germany’s Excellence Strategy EXC-2181/1 - 390900948 (the Heidelberg STRUCTURES Cluster of Excellence).

\section{Preliminaries}
\label{sec:prelim}

\subsection{Lie theory and flag varieties}\label{sec: flag}

References for the preliminary material discussed below include \cite{Hum75}\cite{Helgason01} (Lie theory), \cite{BE89}\cite{BGG82} (flag varieties), \cite{bjorner-brenti} (Bruhat order), and \cite[Section 4.5]{GW12} (symmetric parabolic subgroups).
Sections 2.1 and 3 of \cite{DS20} discuss most of the material from this section with additional detail and examples.

Let $\hat{G}$ be a connected complex semisimple Lie group.
If $\hat{G}$ acts transitively on a connected compact complex manifold $\F,$ then $\F$ is called a \emph{complex flag variety}.  If $G:=\textnormal{Aut}_{0}(\F)$ is the connected component of the identity in the automorphism group of $\F,$ then $G$ is a connected complex semisimple Lie group with trivial center and there is a $G$-equivariant isomorphism $G/P_{D}\simeq \F$ where $P_{D}<G$ is a parabolic subgroup.
Every complex flag variety $\F$ is a smooth complex projective variety.

While we will usually consider the automorphism group $G$ of a flag variety as a complex manifold, it can also be given a compatible linear algebraic group structure (over $\C$).
This follows from the existence of a faithful linear representation ($\Ad$) and \cite[Corollary~7.9]{Borel91}.

Fix Cartan and Borel subgroups $H < B < G$.
Without loss of generality, we can assume that the parabolic subgroup $P_{D}$ satisfies $B<P_{D}$.
Furthermore, associated to these data is the Weyl group $W:=N_{G}(H)/H$, where $N_G(H)$ denotes the normalizer of $H$ in $G$.
The Weyl group is a finite Coxeter group, and we will now recall how its  structure relates to the geometry of the flag variety $\F$.

The Borel subgroup $B<G$ acts on the flag variety $\F$ with finitely many orbits, each of which is bi-holomorphic to a complex affine space.  The parabolic subgroup $P_D < G$ determines a subgroup $W_{D} := ( N_G(H) \cap P_{D})/H$ of the Weyl group $W$ which has the property that the $B$-orbits on $\F$ are in bijection with the coset space $W/W_{D}$: $\F=\sqcup_{w\in W/W_{D}} \mathcal{O}_{w}$.
The closure of an orbit $\overline{\mathcal{O}_{w}}$ is a projective algebraic subvariety of $\F$ which is a union of $\mathcal{O}_{w}$ and other orbits of strictly smaller dimension.
These projective algebraic subvarieties obtained as $B$-orbit closures are called \emph{Schubert varieties}.  The orbit closure relation induces a partial order on $W/W_{D}$.
Since there are finitely many orbits, there exists a unique open dense orbit.
Combinatorially, this implies that the partial order on $W/W_{D}$ has a unique maximal element.  

Now, suppose $\F=G/B$ is a \emph{complete} flag variety.  In this case, the subgroup $W_{D}<W$ is trivial and the induced partial order on $W$ is called the (strong) Chevalley-Bruhat order.  Let $w\mapsto \mathcal{O}_{w}$ denote the bijection between $B$-orbits and elements of $W.$  Recording the dimension of the affine space $\mathcal{O}_{w}$ defines the \emph{length function}
$\ell: W\rightarrow \mathbb{Z}$ on the Coxeter group $W$.
There is a unique element $w_0 \in W$ of maximal length, and the coset $w_0 W_{D}$ is the unique maximal element of $W / W_{D}$.

Finally, the longest element $w_{0}\in W$ acts on $W$ by conjugation.  A parabolic subgroup $P$ such that $B<P$ is \emph{symmetric} if the corresponding subgroup $W_{P}<W$ is invariant under this action.  Equivalently, $P$ is symmetric if given any parabolic subgroup $Q$ such that the intersection $Q\cap P$ is a reductive subgroup of $Q$ and $P,$ then $Q$ is conjugate to $P.$  

\subsection{Cartan projection}\label{sec:cartan-proj}

The Cartan projection appears in the definition of an Anosov homomorphism in the next section; however, it is not used anywhere else in the paper.
We include its definition here for completeness.
See \cite[Section 2.3]{GGKW15} for details.

Let $K<G$ be a maximal compact subgroup such that $K\cap H<K$ is a maximal torus.  Let $\g= \mathfrak{k} \oplus \mathfrak{m}$ be the corresponding Cartan decomposition and $\mathfrak{a}\subset \mathfrak{m}$ a Cartan subspace.  Let $\Pi \subset \mathfrak{a}^{\star}$ be a set of simple restricted roots.  Recall that standard parabolic subgroups of $G$ are in bijection with subsets of the simple restricted roots.

Let $\overline{\mathfrak{a}}^{+}\subset \mathfrak{a}$ be a closed positive Weyl chamber and $A=\textnormal{exp}(\overline{\mathfrak{a}}^{+}).$  Then we have the group level Cartan decomposition $G=KAK$ and an associated map
$a: G\rightarrow \overline{\mathfrak{a}}^{+}$ which takes $g=k_1 \exp(X_g)\, k_2$ to $X_g \in \overline{\mathfrak{a}}^+$.
The map $a$ is called the \emph{Cartan projection}.

\subsection{Anosov homomorphisms and domains of discontinuity}

In this section, we fix a flag variety $\F$ and $G=\textnormal{Aut}_{0}(\F).$  Let $\Gamma$ be a finitely generated torsion-free word hyperbolic group, and let $\Hom(\Gamma,G)$ denote the set of all homomorphisms from $\Gamma$ to $G$ equipped with the compact-open topology.
The group $G$ acts on $\Hom(\Gamma,G)$ by conjugation.

\begin{defn}
Let $P_{A}<G$ be a symmetric parabolic subgroup and $\theta_{A}\subset \Pi$ the corresponding subset of simple restricted roots.  A homomorphism $\rho: \Gamma\rightarrow G$ is $P_{A}$-Anosov if there exist $\kappa_{1}, \kappa_{2}>0$ such that
\begin{align}
\langle \alpha, a(\rho(\gamma)) \rangle \geq \kappa_{1} \lvert \gamma \rvert - \kappa_{2}
\end{align}
for all $\alpha\in \theta_{A}$ and $\gamma\in \Gamma,$ where $\lvert \gamma \rvert$ is the word length on $\Gamma$ with respect to some finite generating set and $\langle \ , \ \rangle$ is the duality pairing.
\end{defn}
There are myriad equivalent definitions of Anosov homomorphisms in the literature, and the one given above seems most economical for our purposes.  It combines the work of Kapovich-Leeb-Porti \cite{KLP14.1}, Guichard-Gu\'{e}ritaud-Kassel-Wienhard \cite{GGKW15} and Bochi-Potrie-Sambarino \cite{BPS19}.  This class of homomorphisms was originally introduced by Labourie \cite{LAB06} (using a rather different definition, and in a more restricted setting) who coined the term Anosov homomorphism.  Also, note that the definition of Anosov homomorphism makes sense for any parabolic subgroup $Q_{A}$, not necessarily symmetric.  But, given any $Q_{A}$-Anosov homomorphism $\rho$, there is a canonical symmetric parabolic subgroup $P_{A}<Q_{A}$ for which $\rho$ is $P_{A}$-Anosov (see \cite{GW12}).

We record here the following essential properties of $P_{A}$-Anosov homomorphisms: these properties were first developed by Labourie \cite{LAB06} in certain special cases, and then proved in general by Guichard-Wienhard \cite{GW12}.
\begin{itemize}
    \item The notion of $P_A$-Anosov only depends on the conjugacy class of the parabolic subgroup $P_A<G.$  Moreover, if $P_{A}^{\prime}$ is a parabolic subgroup such that $P_{A}<P_{A}^{\prime},$ then every $P_{A}$-Anosov homomorphism is $P_{A}^{\prime}$-Anosov.
    \item The set of $P_A$-Anosov homomorphisms is open in $\Hom(\Gamma,G)$ and invariant under the action of $G$ by conjugation.
    \item Every $P_{A}$-Anosov homomorphism is discrete and injective, and 
    \item Given a $P_{A}$-Anosov homomorphism $\rho : \Gamma \to G$, there is an associated continuous, injective, $\rho$-equivariant \emph{limit map} $\xi_\rho : \partial \Gamma \to G/P_A$ and the assignment $\rho \mapsto \xi_{\rho}$ is continuous.  Moreover, this limit map is unique. We will denote the image of $\xi_{\rho}$ by $\Xi_{\rho}\subset G/P_{A}$ and call it the \emph{limit curve} of $\rho.$  
\end{itemize}
Since the symmetric parabolic subgroup $P_A<G$ will be fixed in the discussion to follow, for brevity we will use the term \emph{Anosov} to mean $P_A$-Anosov.  We now recall two central examples of Anosov homomorphisms.  For the statements below, we refer the reader to \autoref{Afam} for a discussion of the complex analytic structure on $\Hom(\Gamma, G).$  

Fix $S$ a closed oriented surface of genus $g\geq 2$ and let $\Gamma=\pi_{1}(S)$ be the corresponding surface group.  A homomorphism $\rho_{0}: \Gamma \rightarrow \PSL(2, \C)$ is called quasi-Fuchsian if there exists a continuous, $\rho_{0}$-equivariant injective map $\xi_{0}: \partial\Gamma\rightarrow \CP^{1}$.  Recall that the Teichm\"{u}ller space $\mathcal{T}(S)$ is the space of oriented marked complex structures on $S$; it is a contractible complex manifold of dimension $3g-3.$  We denote the surface with the opposite orientation by $\overline{S}.$  The following proposition summarizes some properties of quasi-Fuchsian representations that we will need; it combines various results from the literature (most of them well-known) and precise references are given in the proof that follows.
\begin{prop}\label{prop: qf}
Let $\mathcal{QF}_{S}\subset \Hom(\Gamma, \PSL(2, \C))$ be the subset of quasi-Fuchsian homomorphisms.  Then,
\begin{rmenumerate}
\item Every $\rho\in \mathcal{QF}_{S}$ is Anosov.
\item $\mathcal{QF}_{S}$ has the structure of a connected complex manifold of dimension $6g-3.$
\item There is a principal $\PSL(2,\C)$ bundle $\mathcal{QF}_{S}\rightarrow \mathcal{T}(S)\times \mathcal{T}(\overline{S}).$
\item For $\rho \in \mathcal{QF}_{S}$, the image $\rho(\Gamma)$ has trivial centralizer in $\PSL_2\C$.
\end{rmenumerate}
\end{prop}
\begin{proof}
Every $\rho\in \mathcal{QF}_{S}$ is Anosov by \cite[Theorem 5.15]{GW12}, giving (i).

For (iv), it will be convenient to use the algebraic structure of $\PSL(2,\C)$.
First we observe that $\rho(\Gamma)$ is Zariski dense in $\PSL(2,\C)$:
The Zariski closure is a complex Lie subgroup of $\PSL(2,\C)$, and any \emph{proper} subgroup of this type acts on $\CP^1$ with a global fixed point.
Thus it suffices to show $\rho(\Gamma)$ acts on $\CP^1$ without fixed points.
Take elements $g,g'$ of $\Gamma$ whose attracting and repelling fixed points on $\partial \Gamma$ are distinct (e.g.~elements corresponding to homotopically distinct simple closed curves on $S$).
Then $\rho(g)$ and $\rho(g')$ have no common fixed points on $\CP^1$ (by injectivity and dynamics-preserving properties of $\xi_0$), as required.

The centralizer of any subset $E \subset \PSL(2,\C)$ also centralizes the Zariski closure of $E$, since the conjugation action of $\PSL(2,\C)$ on itself is algebraic.
Therefore, the centralizer of $\rho(\Gamma)$ lies in the center of $\PSL(2,\C)$, which is trivial.
This proves (iv).

In \cite[Section 1.4]{GOL84}, Goldman shows that for $i=0,2$, the dimension of $H^{i}(\Gamma, \sl(2,\C)_{\rho})$ is equal to that of the Lie algebra of the centralizer $\rho(\Gamma)$.
We have seen that this centralizer is trivial, hence both of these cohomology groups vanish.
In \cite{GOL84} it is shown that this implies that $\rho$ is a smooth point of $\Hom(\Gamma, \PSL(2, \C))$.
Since being quasi-Fuchsian is an open property in the space of homomorphisms, this implies $
\mathcal{QF}_{S}$ is a smooth open subset of $\Hom(\Gamma, \PSL(2, \C)).$   

Finally, $\PSL(2, \C)$ acts freely and properly on $\mathcal{QF}_{S}$ with quotient isomorphic to $\mathcal{T}(S) \times \mathcal{T}(\overline{S})$ by the Bers Simultaneous Uniformization Theorem \cite{BER60}.  This proves (ii) and (iii).
\end{proof}

Let $G$ be a complex simple Lie group of adjoint type and $\iota_{G}: \PSL(2, \C)\rightarrow G$ the Kostant principal three dimensional subgroup \cite{KOS59}.  A homomorphism $\rho: \Gamma\rightarrow G$ is $G$-quasi-Fuchsian if there exists a quasi-Fuchsian homomorphism $\rho_{0}$ such that $\rho$ is conjugate to $\iota_{G}\circ \rho_{0}.$  It follows from \cite{DS20}
 that every $G$-quasi-Fuchsian homomorphism is $B$-Anosov where $B<G$ is a Borel subgroup.  
\begin{prop}\label{prop:GQF}
Let $\mathcal{QF}_{S}(G)\subset \Hom(\Gamma, G)$ denote the subset of $G$-quasi-Fuchsian homomorphisms.  Then,
\begin{rmenumerate}
\item $\mathcal{QF}_{S}(G)$ admits the structure of a connected complex manifold of dimension $6g-6+\textnormal{dim}(G).$
\item The subset $\mathcal{QF}_{S}(G)\subset \Hom(\Gamma, G)$ is a locally closed complex analytic subspace.  Moreover, every $\rho\in \mathcal{QF}_{S}(G)$ is a smooth point of $\Hom(\Gamma, G).$ 
\item Finally, $G$ acts properly and freely on $\mathcal{QF}_{S}(G)$ and the quotient $\underline{\mathcal{QF}}_{S}(G)$ is contractible.
\item For each $\rho \in \mathcal{QF}_{S}(G)$, the image $\rho(\Gamma)$ has trivial centralizer in $G$.
\end{rmenumerate}
\end{prop}
\begin{proof}
Consider the injective holomorphic map $\iota_{G}: \mathcal{QF}_{S}\rightarrow \mathcal{QF}_{S}(G)$ induced by $\iota_{G}.$  Let $Z\subset \mathcal{QF}_{S}(G)$ denote the image.  Then, $G$ acts properly on $Z$ and the action map $G\times Z\rightarrow \mathcal{QF}_{S}(G)$ is surjective with fibers isomorphic to $\PSL(2, \C).$  This proves (i), and implies that the subset $\mathcal{QF}_{S}(G)\subset \Hom(\Gamma, G)$ is a locally closed complex analytic subspace.

For (iv), we use an argument similar to the proof of \autoref{prop: qf}(iv).
Let $Z_\rho$ denote the centralizer in $G$ of $\rho(\Gamma)$.
As in \autoref{prop: qf}, $Z_\rho$ also centralizes the Zariski closure of $\rho(\Gamma)$.
If we write $\rho = g (\iota_G \circ \rho_0) g^{-1}$ for $g \in G$ (as the definition of $\mathcal{QF}_{S}(G)$ allows), we first claim that the Zariski closure of $\rho(\Gamma)$ is equal to $g \iota_G(\PSL(2,\C)) g^{-1}$.
Since this (closed) algebraic subgroup of $G$ contains $\rho(\Gamma)$, it also contains the Zariski closure of $\rho(\Gamma)$.
If the Zariski closure of $\rho(\Gamma)$ were a proper subgroup of $g \iota_G(\PSL(2,\C)) g^{-1}$, then conjugating that group by $g^{-1}$ and taking the preimage by $\iota_G$ would give a complex algebraic subgroup of $\PSL(2,\C)$ containing the quasi-Fuchsian group $\rho_0(\Gamma)$.
But, as observed in the proof of \autoref{prop: qf} quasi-Fuchsian groups are Zariski dense, so this is a contradiction.
Therefore, $Z_\rho$ centralizes $g \iota_G(\PSL(2,\C)) g^{-1}$.
Finally, the centralizer of $\iota_G(\PSL(2,\C))$ or any conjugate thereof is trivial, as follows from the structure theory for the principal three-dimensional subgroup developed in \cite[Section 5]{KOS59}.

By a theorem of Goldman \cite{GOL84}, triviality of $Z_\rho$ implies that $\rho$ is a smooth point of $\Hom(\Gamma, G)$, which completes the proof of (ii).  

Finally, the action map $G\times Z\rightarrow \mathcal{QF}_{S}(G)$ induces an isomorphism $\underline{\mathcal{QF}}_{S}\simeq \underline{\mathcal{QF}}_{S}(G)$, and therefore $\underline{\mathcal{QF}}_{S}(G)$ is contractible by \autoref{prop: qf}, and (iii) follows.
\end{proof}

Next, consider the inclusion $\PSL(2, \R)< \PSL(2, \C).$  A quasi-Fuchsian homomorphism $\rho_{0}: \Gamma\rightarrow \PSL(2, \C)$ is \emph{Fuchsian} if it can be conjugated to lie in $\PSL(2, \R).$  The set of Fuchsian homomorphisms is denoted by $\mathcal{F}_{S}\subset \Hom(\Gamma, \PSL(2, \C)).$  The restriction of the Kostant homomorphism $\iota_{G}: \PSL(2, \R)\rightarrow G_{\R}$ takes values in a split real form $G_{\R}<G.$  A homomorphism $\rho: \Gamma\rightarrow G$ is $G$-Fuchsian if there is a Fuchsian homomorphism $\rho_{0}: \Gamma\rightarrow \PSL(2, \R)$ such that 
$\rho$ is conjugate to $\iota_{G}\circ \rho_{0}: \Gamma\rightarrow G_{\R}.$ 
We denote the set of $G$-Fuchsian homomorphisms by $\mathcal{F}_{S}(G)\subset \Hom(\Gamma, G).$    

A homomorphism $\rho: \Gamma\rightarrow G_{\R}$ is $G_{\R}$-Fuchsian if there exists a Fuchsian homomorphism $\rho_{0}: \Gamma\rightarrow \PSL(2, \R)$ such that $\rho$ is conjugate via an element of $G_{\R}$ to $\iota_{G}\circ \rho_{0}.$  The set of all $G_{\R}-$ Fuchsian homomorphisms is either connected in $\Hom(\Gamma, G_{\R}),$ or has two connected components.  A homomorphism is $G_{\R}$-\emph{Hitchin} if it lies in a component of $\Hom(\Gamma, G_{\R})$ which contains the $G_{\R}$-Fuchsian homomorphisms.  A homomorphism $\rho: \Gamma\rightarrow G$ is called $G$-Hitchin if it is conjugate to a $G_{\R}$-Hitchin homomorphism $\eta: \Gamma\rightarrow G_{\R}.$  We denote the set of $G$-Hitchin homomorphisms by $\mathcal{H}_{S}(G)\subset \Hom(\Gamma, G)$  and the corresponding set of $G_{\R}$-Hitchin homomorphisms by $\mathcal{H}_{S}(G_{\R})\subset \Hom(\Gamma, G_{\R}).$

The following theorem follows readily from \cite{HIT92}.  
\begin{prop}\label{prop:GHIT}
The set $\mathcal{H}_{S}(G)\subset \Hom(\Gamma, G)$ of $G$-Hitchin homomorphisms has the structure of a smooth real manifold of real dimension $(2g-2)\cdot \textnormal{dim}(G_{\R}) + \textnormal{dim}_{\R}(G).$  Moreover, every $\rho\in \mathcal{H}_{S}(G)$ is a smooth point of $\Hom(\Gamma, G).$  Finally, $G$ acts properly and freely on $\mathcal{H}_{S}(G)$ and the quotient is contractible.
\end{prop}
\begin{proof}
It was proved by Hitchin \cite{HIT92} that $\mathcal{H}_{S}(G_{\R})\subset \Hom(\Gamma, G_{\R})$ is an open smooth submanifold of dimension $(2g-2)\cdot \textnormal{dim}(G_{\R}) + \textnormal{dim}(G_{\R}).$  Now consider the action map
$G\times \mathcal{H}_{S}(G_{\R}) \rightarrow \mathcal{H}_{S}(G).$  This map is smooth and surjective, and the fibers are isomorphic to $G_{\R}.$  This proves that $\mathcal{H}_{S}(G)$ is a smooth real manifold of real dimension $(2g-2)\cdot \textnormal{dim}(G_{\R}) + \textnormal{dim}_{\R}(G).$

The final statement follows from Hitchin \cite{HIT92}.
\end{proof}
In the foundational paper \cite{LAB06}, Labourie proved that every $\PSL(n, \R)$-Hitchin homomorphism is Anosov with respect to the Borel subgroup in $\PSL(n, \C).$  By basic representation theoretic considerations, this implies that every $G$-Hitchin homomorphism is $B$-Anosov when $G$ is of type $B_{n}, C_{n}$ and $G_{2}.$  It was later clarified (see \cite{FG06}) that every $G$-Hitchin homomorphism is Anosov with respect to the Borel $B<G.$ 

\subsection{Domains of discontinuity}
\label{subsec:domains}

Given an Anosov homomorphism $\rho: \Gamma\rightarrow G,$ a systematic theory has been developed by Kapovich-Leeb-Porti \cite{KLP13}, initiated by Guichard-Wienhard \cite{GW12}, which produces open subsets of flag varieties of $G$ on which $\Gamma$ acts freely, properly discontinuously, and cocompactly.
We recall some essential points of this theory, referring the reader to \cite{KLP13}\cite{DS20} for details.

A subset $I \subset W$ of the Weyl group $W$ that is convex for the \cv order and which contains the identity is an \emph{ideal} (i.e. $b\in I$ implies $a\in I$ for all $a\leq b$).  Let $w_{0}\in W$ denote the longest element.
An ideal is \emph{balanced} if $W = I \cup w_0I$ and $I \cap w_0I = \emptyset$.

Let $P_{D}<G$ be a parabolic subgroup and $W_{D}<W$ the corresponding subgroup: if $I\subset W$ is an ideal which is right-$W_{D}$ invariant, then the discussion in \autoref{sec: flag} implies it corresponds to a finite union of Schubert varieties in $\F\simeq G/P_{D}$ indexed by the cosets $I/W_{D}.$  

Let $P_{A}<G$ be a symmetric parabolic subgroup and consider the subgroup $W_{A}<W.$  An ideal which is left $W_{A}$-invariant and right $W_{D}$-invariant will be called an ideal of type $(P_{A}, \F).$  Given an ideal of type $(P_{A}, \F),$ the corresponding finite union of Schubert varieties in $\F$ is also a union of $P_{A}$-orbits in $\F.$

Given an ideal $I\subset W$ of type $(P_{A}, \F),$ we call the corresponding union of Schubert varieties the \emph{model thickening} and denote it by $\Phi_{o}^{I}\subset \F.$ More generally, to any $x=gP_{A}\in G/P_{A},$ we can associate the $g$-translate of the model thickening $\Phi_{x}^{I}:=g\cdot \Phi_{o}^{I}.$  Note that this is well defined since the model thickening $\Phi_{o}^{I}$ is left $P_{A}$-invariant.

Now, let $\rho: \Gamma\rightarrow G$ be a $P_{A}$-Anosov homomorphism with limit map
$\xi_{\rho}: \partial \Gamma\rightarrow G/P_{A}.$  Let $I\subset W$ be an ideal of type $(P_{A}, \F).$  The limit set $\Lambda_{\rho}^{I}\subset \F$ associated to these data is defined by
\begin{align}
\Lambda_{\rho}^{I}:=\bigcup_{t \in \partial \Gamma} \Phi_{\xi_{\rho}(t)}^{I}.
\end{align}
The limit set $\Lambda_{\rho}^{I}$ is a closed set, and therefore the complement
$\Omega_{\rho}^{I}:=\F \backslash \Lambda_{\rho}^{I}$ is open.  We can now state the fundamental theorem of Kapovich-Leeb-Porti \cite{KLP13}.

\begin{thm}\label{thm: KLP}
Let $\rho: \Gamma\rightarrow G$ be $P_{A}$-Anosov, where $G = \Aut(\F)$, and let $I\subset W$ be a balanced ideal of type $(P_{A}, \F).$  If $\Omega_{\rho}^{I}$ is non-empty, then the discrete group $\rho(\Gamma)$ acts freely, properly discontinuously, and cocompactly on $\Omega_{\rho}^{I}.$  
\end{thm}
We will refer to such a pair $(\rho, I)$ as a thickened Anosov homomorphism of type $(P_{A}, \F).$  

By \autoref{thm: KLP}, if $\Omega_{\rho}^{I}$ is non-empty, then the quotient
$\W_{\rho}^{I}:=\rho(\Gamma)\backslash \Omega_{\rho}^{I}$ is a compact complex manifold.  The goal of this paper is to study the relationship, in terms of deformation theory, between the Anosov homomorphism $\rho$ and the complex manifold $\W_{\rho}^{I}.$  

\subsection{Hausdorff dimension of limit sets}
\label{subsec:hdim}

As mentioned in the introduction, our results depend on the limit set $\Lambda_{\rho}^{I}$ being ``small''.  Let $(\rho, I)$ be a thickened Anosov homomorphisms of type $(P_{A}, \F).$ Given an integer $k\geq 0,$ we say that $(\rho, I)$ is \emph{$k$-small} if the limit set $\Lambda_\rho^{I}\subset \F$ is a null set for any Riemannian Hausdorff measure on $\F$ of codimension $k$, i.e.~if
\[
\mathcal{H}_{2N-k}(\Lambda_\rho^{I})=0, \text{where } N=\dim_\C \F.
\]
The notion of a null set is independent of the choice of Riemannian metric.

In particular, a pair $(\rho, I)$ as above is $k$-small if the Hausdorff dimension of its limit set is strictly less than $2N-k$.
Note that the larger $k$ is, the smaller the limit set of a $k$-small thickened Anosov homomorphism is required to be.
Our main results apply to $4$-small homomorphisms.

This condition is only interesting if it can be shown that $k$-small homomorphisms exist.
Fortunately, this can be done by bounding the Hausdorff dimension of the limit curve and model thickening separately.  Recall that the length $\ell_{D}: W/W_{D}\rightarrow \mathbb{Z}$ is defined to be the dimension of the Schubert cell in $\F$ indexed by the given element of $W/W_{D}.$  Given a balanced ideal $I\subset W$ of type $(P_{A}, \F),$ we define its length $\ell_{D}(I)$ to be the maximum value of the function $\ell_{D}: I/W_{D}\rightarrow \mathbb{Z}.$  
By \cite[Theorem 4.7]{DS20}, the Hausdorff dimension of the limit set satisfies
\[ \label{hdim ineq}
\hdim \Lambda_\rho^{I} \leq 2\ell(I) + \hdim \Xi_\rho
\]
and thus it follows immediately:
\begin{prop}
Let $(\rho, I)$ be a thickened Anosov homomorphism of type $(P_{A}, \F).$
If $\hdim \Xi_\rho < 2(N - \ell(I)) - k$, then $(\rho, I)$ is $k$-small.
\end{prop}
Depending on the balanced ideal $I$, the quantity $N-\ell(I)$ can be as large as $\left \lfloor \frac{N}{2} \right \rfloor$, but in \cite[Theorem 4.1(ii)]{DS20} it is shown that this difference is at least $3$ when $G$ has no factors of small rank (see the above-cited reference for the precise list of low-rank exceptions).
Thus for such $G$, any homomorphism with $\hdim \Xi_\rho < 2$ is $4$-small independent of the balanced ideal.

We record the following important result for surface groups which combines the results of the authors \cite{DS20} and those of Pozzetti-Sambarino-Wienhard \cite{PSW192}\cite{PSW19}.

\begin{thm}\label{thm: PSW}
Let $S$ be a closed orientable surface, $\Gamma=\pi_{1}(S)$ a closed surface group, and $G$ a complex simple Lie group of adjoint type.
\begin{rmenumerate}
\item If $\rho: \Gamma \rightarrow G$ is $G$-quasi-Fuchsian, then the limit curve $\Xi_{\rho}\subset G/B$ satisfies $\hdim(\Xi_{\rho})<2.$
\item Suppose $G$ is not of type $F_{4}, E_{6}, E_{7}, E_{8}.$  If $\rho: \Gamma\rightarrow G_{\R}$ is $G$-Hitchin, then the limit curve $\Xi_{\rho}\subset G/B$ satisfies $\hdim(\Xi_{\rho})=1.$
\item There exists a connected open $G$-invariant subset $\mathcal{U}\subset \A\subset \Hom(\Gamma, G)$ containing the sets of $G$-quasi-Fuchsian such that the function
\begin{align}
\hdim: \mathcal{U}&\rightarrow \R \\
                              \rho &\mapsto \hdim(\Xi_{\rho})
\end{align}
is continuous, $G$-invariant, and satisfies $\hdim(\rho) < 2$ for all $\rho\in \mathcal{U}.$

If $G$ is not of type $F_{4}, E_{6}, E_{7}$ or $E_{8},$ then $\mathcal{U}$ can be chosen to also contain the set of $G$-Hitchin homomorphisms.
\item Suppose $G$ is not of type $A_{1}, A_{2}, A_{3}, B_{2}$ and $\rho: \Gamma\rightarrow G$ is in the neighborhood $\mathcal{U}\subset\A$ of $G$-quasi-Fuchsian homomorphisms provided by (iii). Then for every balanced ideal $I$ of type $(P_{A}, \F),$ the thickened Anosov homomorphism $(\rho, I)$ is $4$-small.
\item Suppose $G$ is not of type $A_{1}, A_{2}, A_{3}, B_{2}, F_{4}, E_{6}, E_{7}$ or $E_{8}.$ Let $\rho: \Gamma\rightarrow G$ be in the neighborhood $\mathcal{U}\subset\A$ of $G$-quasi-Fuchsian and $G$-Hitchin homomorphisms provided by (iii).  Then, for every balanced ideal $I$ of type $(P_{A}, \F),$ the thickened Anosov homomorphism $(\rho, I)$ is $4$-small.
\end{rmenumerate}
\end{thm}
\begin{remark}
The exclusion of types $A_{1}, A_{2}, A_{3}$ and $B_{2}$ above are essential to state uniform results.  Meanwhile, the exclusion of the exceptional types $F_{4}, E_{6}, E_{7}$ and $E_{8}$ is likely to be unnecessary.  We have been informed of forthcoming work by Sambarino which extends (ii) and (iii) to these exceptional groups, in which case they can be removed from the list of exceptions in (v) as well. 
\end{remark}
\begin{proof}
Item (i) is proved in \cite{DS20}.  Items (ii) and (iii) follow from the work of Pozzetti-Sambarino-Wienhard \cite{PSW19}\cite{PSW192}.  

Finally, it is proved in \cite{DS20} that if $G$ is not of type $A_{1}, A_{2}, A_{3}$ or $B_{2},$ then every balanced ideal $I$ of type $(P_{A}, \F)$ satisfies $\ell(I)\leq N-3$ where $N$ is the complex dimension of $\F.$  Therefore, given any $\rho:\Gamma\rightarrow G$ such that $\hdim(\Xi_{\rho})<2$ and using inequality \eqref{hdim ineq} we obtain
\begin{align}
\hdim \Lambda_{\rho}^{I}\leq 2\ell(I) + \hdim \Xi_{\rho}< 2(N-3)+2=2N-4.
\end{align}
Hence, $(\rho, I)$ is $4$-small.  Hence, (i), (ii), and (iii) imply (iv) and (v).
\end{proof}

\section{The Anosov family}
\label{sec:anosov}

\subsection{Complex analytic spaces}

In this paper, we will assume some familiarity with the basic theory of complex analytic spaces and coherent sheaves.
We refer the reader to the books of Fischer \cite{FIS76} and Grauert-Remmert \cite{GR84} for details.  Below, we give a rapid review of the basic concepts we will use.

A complex analytic space is a $\C$-locally ringed Hausdorff space $(X, \mathcal{O}_{X})$ locally isomorphic to the vanishing locus of finitely many holomorphic functions on an open subset $U\subset \C^{k_{U}}$
equipped with its structure sheaf.

We call a morphism $(f, f^{\sharp}): (X, \mathcal{O}_{X})\rightarrow (Y, \mathcal{O}_{Y})$ between complex analytic spaces a holomorphic map and usually we suppress the sheaf morphism $f^{\sharp}: \mathcal{O}_{Y}\rightarrow f_{\star}\mathcal{O}_{X}$ and just say that $f: X\rightarrow Y$ is a holomorphic map.

Given a complex analytic space $(X, \mathcal{O}_{X}),$ a point $x\in X$ is a smooth point if there exists a connected open neighborhood $U\subset X$ of $x$ such that $(U, \mathcal{O}_{X}|_{U})$ is isomorphic to a complex manifold.  The set of smooth points is open in $X$.

The tangent sheaf $\Theta_{X}$ of a complex analytic space $(X, \mathcal{O}_{X})$ is the sheaf of $\C$-linear derivations of the structure sheaf $\mathcal{O}_{X}.$  The tangent sheaf $\Theta_{X}$ is a coherent sheaf of $\mathcal{O}_{X}$-modules.

Next, we define a notion of a smooth holomorphic map in this setting:
\begin{defn}\label{def: smooth}
Let $f: X\rightarrow Y$ be a holomorphic map.  Then we say $f$ is \emph{smooth} if:
\begin{itemize}
\item For all $x\in X,$ there exists an open neighborhood $U\subset X$ of $x$ and an open neighborhood $V\subset Y$ of $f(x)$ such that $f(U)=V.$  
\item There exists an open subset $W\subset \C^{k}$ where $k$ is a non-negative integer and an isomorphism $U\simeq V\times W$ such that the diagram
\begin{center}
\begin{tikzcd}
U \arrow{r}{\simeq} \arrow{d}{f}
& V\times W \arrow{dl} \\
V
\end{tikzcd}
\end{center}
commutes.
\end{itemize}
\end{defn}
If $f: X\rightarrow Y$ is a holomorphic map between complex manifolds, then $f$ is smooth in this sense if and only if it is a submersion.   The composition of smooth maps is smooth.

Finally, suppose we are given complex analytic spaces $X,Y,Z$ and holomorphic maps $f: X\rightarrow Z$ and $g: Y\rightarrow Z.$  Then, there exists a (unique, up to unique isomorphism) complex analytic space $X\times_{Z} Y$ and 
a commutative diagram of holomorphic maps
\begin{center}
\begin{tikzcd}
X\times_{Z} Y \arrow{r} \arrow{d}
& Y \arrow{d}{g}\\
X \arrow{r}{f}
& Z.
\end{tikzcd}
\end{center}
The space $X\times_{Z} Y$ is called the fiber product and the projection $X\times_{Z} Y\rightarrow X$ is called the base change.  If $g$ is a smooth holomorphic map, then the base change $X\times_{Z} Y\rightarrow X$ is smooth, and in this setting we will use the notation $f^{\star}Y:=X\times_{Z} Y$ and call $f^{\star}Y$ the pullback of the smooth holomorphic map $g: Y\rightarrow Z$ via $f.$  

\subsection{Families of complex manifolds}
In this section, we recall the basic setup of complex analytic deformation theory which will be used throughout this paper.
Details can be found in \cite{Palamodov}, for example.

\begin{defn}
A complex analytic family of compact complex manifolds is a triple $(\Y, B, p)$ where $\Y$ and $B$ are complex analytic spaces and
$p: \Y \rightarrow B$ is a proper smooth holomorphic map.
\end{defn}
By the definition of smoothness, the fibers of a complex analytic family are compact complex manifolds.

The complex analytic space $B$ appearing in a complex analytic family $(\Y, B, p)$ is called the \emph{base}.
A \emph{pointed complex analytic family} is defined analogously, but with the base replaced by a pointed space $(B,b)$.

Let $(\Y, B, p)$ be a complex analytic family.  For $b \in B,$ we denote the fiber of $p$ by $\Y_{b}:=p^{-1}(b).$  
\begin{defn}
A complex analytic family $(\Y, B, p)$ is versal at $b \in B$ (or, the pointed family $(\Y,(B,b),p)$ is versal) if:
\begin{itemize}
\item For any pointed complex analytic family $(\Y',B',b')$ equipped with a fixed isomorphism $\Y_b \simeq \Y'_{b'}$, there exists a connected open set $U \subset B'$ with $b' \in U,$ and a holomorphic map $F: U\rightarrow B$ such that $F(b')=b.$
\item There is an isomorphism $F^{\star}\Y\simeq \Y^{\prime}$ extending the isomorphism $\Y_b \simeq \Y'_{b'}.$ 
\end{itemize}
If the above maps are unique, then the family $(\Y, B, p)$ is called universal at $b\in B.$
\end{defn}
Since smoothness is preserved by base change, $F^{\star}\Y$ is a complex analytic family over $U$.

As introduced in the previous section, the tangent sheaf of a complex analytic space $X$ is denoted by $\Theta_{X}$.  We record here the fundamental theorem of Kuranishi \cite{KUR62} and Grauert \cite{Grauert74}.  See \cite{Balaji} for a presentation and discussion that more closely follows our approach.
\begin{thm}\label{kur family}
Let $X$ be a compact complex manifold.  Then there exists an open set $U\subset H^{1}(X, \Theta_{X})$ containing $0$ and a holomorphic map
\begin{align}
Kr: U\rightarrow H^{2}(X, \Theta_{X})
\end{align}
satisfying $Kr(0)=0,$ such that $Kr^{-1}(0):=\mathcal{B}$ is the base of a versal complex analytic family $(\Y, (\mathcal{B},0), p)$ with $\Y_{0}=X.$  If $H^0(X, \Theta_{X})=\{0\},$ then the family $(\Y, (\mathcal{B},0), p)$ is universal.
\end{thm}
The family provided by \autoref{kur family} is called the \emph{Kuranishi family}, and $Kr$ is the \emph{Kuranishi map}.  If the Kuranishi map is zero, then the base of the Kuranishi family is a connected open neighborhood of $0\in H^{1}(X, \Theta_{X}),$ and is therefore smooth.

If $(\Y, B, p)$ is a complex analytic family, then smoothness of the map $p: \Y\rightarrow B$ implies that the induced map $\Theta_{\Y}\rightarrow p^{\star}\Theta_{B}$ is surjective. Hence, there is an exact sequence of sheaves of $\O_{\Y}$-modules
\begin{equation}
\label{eqn:pre-ks-seq}
0\rightarrow \Theta_{\Y / B} \rightarrow \Theta_{\Y} \rightarrow p^{*}\Theta_{B}\rightarrow 0
\end{equation}
where the \emph{vertical} tangent sheaf $\Theta_{\Y/B}$ is the kernel of the map $\Theta_{\Y}\rightarrow p^{\star}\Theta_{B}.$  Furthermore, since $p$ is smooth, the sheaf $\Theta_{\Y/B}$ is locally free.

Given $b \in B$, let $\iota_{b}: \Y_{b}\rightarrow \Y$ denote the inclusion.  Applying $\iota_{b}^{\star}$ to \eqref{eqn:pre-ks-seq} yields a short exact sequence of sheaves
of $\O_{\Y_{b}}$-modules
\begin{align}\label{eqn:ks-seq}
0\rightarrow \Theta_{\Y_{b}}\rightarrow \iota_{b}^{\star} \Theta_{\Y}\rightarrow \iota_{b}^{\star}p^{\star} \Theta_{B}\rightarrow 0.
\end{align}
There is an isomorphism $\iota_{b}^{\star}p^{\star}\Theta_{B} \simeq T_{b}\B \otimes_{\C} \O_{\Y_{b}},$ and since $\Y_{b}$ is compact, this implies $H^{0}(\Y_{b}, \iota_{b}^{\star}p^{\star}\Theta_{B})\simeq T_{b}B.$ Therefore, taking the connecting homomorphism in the long exact sequence in sheaf cohomology associated to \eqref{eqn:ks-seq} yields a complex linear map
\begin{align}
\KS_{b}: T_{b}B \rightarrow H^{1}(\Y_{b}, \Theta_{\Y_{b}})
\end{align}
which is called the \emph{Kodaira-Spencer map} of the pointed family $(\Y,(B,b),p)$.
The pointed family is called \emph{effective} if $\KS_{b}$ is injective, and \emph{complete} if $\KS_{b}$ is surjective.  

We say that a family $(\Y,B,p)$ is effective, complete, or universal if the corresponding condition holds for every associated pointed family, that is, for every $b \in B$.

\begin{prop}
\label{prop:universality-criterion}
Let $(\Y, (B,b), p)$ be a pointed complex analytic family that is effective and complete.
(That is, suppose $\KS_b$ is an isomorphism.)
Furthermore suppose $b$ is a smooth point of $B.$
Then the pointed family $(\Y, (B,b), p)$ is versal.  If $H^{0}(\Y_{b}, \Theta_{\Y_{b}})=\{0\},$ then the pointed family $(\Y, (B,b), p)$ is universal.
\end{prop}

\begin{proof}
Let $(\mathcal{B},0)$ be the base of the Kuranishi family.  By \autoref{kur family}, the Kuranishi family is versal.  Hence, there exists a pointed open set $(U,b) \subset (\mathcal{B},b)$ and a pointed holomorphic map
\begin{align}
F: (U,b)\rightarrow (\mathcal{B}, 0)
\end{align}
such that the pullback of the Kuranishi family along $F$ is isomorphic to the restriction of $\Y$ to $U.$  

Composing with the inclusion $\iota: \mathcal{B} \rightarrow H^{1}(X, \Theta_{X})$ gives a holomorphic map of pointed complex manifolds $\iota\circ F: (U,b)\rightarrow \left ( H^{1}(X, \Theta_{X}), 0 \right )$ such that $d(\iota\circ F): T_{b}B\rightarrow H^{1}(X, \Theta_{X})$ is equal to the Kodaira-Spencer map $\KS_{b}$.  By assumption, $d(\iota\circ F): T_{b}B\rightarrow H^{1}(X, \Theta_{X})$ is an isomorphism, and therefore by the implicit function theorem, upon shrinking $B,$ the map
$\iota\circ F: (U,b)\rightarrow H^{1}(X, \Theta_{X})$ is a biholomorphism onto an open neighborhood of $0\in H^{1}(X, \Theta_{X}).$  

Therefore, the inclusion $\iota: \mathcal{B}\rightarrow H^{1}(X, \Theta_{X})$ is a biholomorphism onto an open neighborhood of $0\in H^{1}(X, \Theta_{X}),$ which implies that $(\mathcal{B}, 0)$ is smooth and 
$F: (U,b)\rightarrow (\mathcal{B}, 0)$ is a biholomorphism.  Since (uni)-versality is an isomorphism invariant, applying \autoref{kur family} the family $(\Y, (B, b), p)$ is versal, and universal if $H^{0}(X, \Theta_{X})=\{0\}.$  This completes the proof.
\end{proof}

\subsection{Families over Anosov homomorphisms}\label{Afam}

Let $\Gamma$ be a word hyperbolic group and $G=\Aut_{0}(\F)$ where $\F$ is a complex flag variety.  Then, $G$ is a connected complex semisimple Lie group of adjoint type, and it also carries the structure of a semisimple affine algebraic group over $\C.$  Throughout this paper, we will always assume that $G$ is of the above form to ensure that $\g\simeq H^{0}(\F, \Theta_{\F}).$

By \cite{SIK12}, the homomorphism space $\Hom(\Gamma, G)$ admits the structure of an affine scheme of finite type over $\C.$  There is a functor, called analytification, from the category of schemes locally of finite type over $\C$ to the category of complex analytic spaces (see \cite{NEE07}): applying this functor to the homomorphism scheme $\Hom(\Gamma, G)$ yields a complex analytic space.  In this paper, whenever we write $\Hom(\Gamma, G)$ we equip it with this complex analytic structure.  Note that the topology on this complex analytic space agrees with the compact-open topology.

Fix $P_{A}<G$ a symmetric parabolic subgroup and let $\A \subset \Hom(\Gamma, G)$ denote the open subset consisting of $P_{A}$-Anosov homomorphisms.  We equip $\A$ with its induced complex analytic structure.
Let $I$ be a balanced ideal of type $(P_{A}, \F).$  We define an open subset $\A_{I}\subset \A$ by the condition $\rho\in \A_{I}$ if and only if $\Omega_{\rho}^{I}$ is non-empty.  Clearly, if $(\rho, I)$ is $k$-small for any $k\geq 0$, then $\rho\in \A_{I}.$

For certain $\mathcal{F}$ and $I$, it may be the case that $\A_{I}$ is empty.  To avoid a nonemptiness hypotheses in many of our statements, it is our standing assumption that whenever we speak of $\A_{I}$ in this paper, it is assumed to be non-empty.  

The \emph{universal domain} is the open set $\Omega^{I}\subset \F \times \A_{I}$ defined by
\begin{align}
\Omega^{I}:=\{ (x, \rho)\in \F\times \A_{I} \ | \ x\in \Omega_{\rho}^{I}\}.
\end{align}
Since $\Omega^{I}$ is open, it acquires the structure of a complex analytic space.  Moreover, since $\F\times \A_{I}\rightarrow \A_{I}$ is smooth, and smoothness is a local property, the map
$\Omega^{I}\rightarrow \A_{I}$ is smooth.

The universal domain admits commuting actions of $G$ and $\Gamma$ given by
\begin{itemize}
\item $g\cdot (x, \rho)=(g\cdot x, g\cdot\rho\cdot g^{-1})$,
\item  $\gamma \cdot(x, \rho)=(\rho(\gamma)\cdot x, \rho)$.
\end{itemize}
The quotient of $\Omega^I$ by the free and proper $\Gamma$-action is denoted by $\W^{I}:= \Gamma \backslash \Omega^{I}$ and the map $\Omega^I \to \A_I$ given by $(x,\rho) \mapsto \rho$ descends to a natural $G$-equivariant projection $p: \W^{I}\rightarrow \A_{I}$.

\begin{thm}
The triple $(\W^{I}, \A_{I}, p)$ is a $G$-equivariant complex analytic family.
Moreover, the map $p: \W^{I}\rightarrow \A_{I}$ is locally trivial as a continuous map.  Finally, if $\rho, \rho^{\prime}\in \A_{I}$ lie in the same connected component, then $\W_{\rho}^{I}$ and $\W_{\rho^{\prime}}^{I}$ are diffeomorphic.
\end{thm}

In the proof we will refer to a result about quotients from the next subsection (\autoref{thm: quotients}).  

\begin{proof}
Since $\W^{I}$ is the quotient of the analytic space $\Omega^{I}$ by a free and proper action of a countable discrete group of automorphisms, \autoref{thm: quotients} implies there is a unique complex analytic structure on $\W^{I}$ such that the quotient map $q: \Omega^{I}\rightarrow \W^{I}$ is a smooth holomorphic map.

Since $\Omega^{I}\rightarrow \A_{I}$ is smooth, this implies $p: \W^{I}\rightarrow \A_{I}$ is smooth and proper.  Therefore, $(\W^{I}, \A_{I}, p)$ is a complex analytic family.  Moreover, since the $G$ and $\Gamma$-actions on $\Omega^{I}$ commute, the $G$-action on $\A_{I}$ clearly lifts to $\W^{I}.$

As noted in \cite[Section 5.2]{DS20}, the topological local triviality of $p$ follows from the proof of \cite[Theorem 9.12]{GW12}.  In \cite[Theorem 5.1]{DS20} it is shown that the diffeomorphism type is constant on each component.
\end{proof}
\begin{remark}
In general, the space $\A$ of Anosov homomorphisms is known to have singular points and non-reduced points. In fact, by results of Kapovich-Millson \cite{KM17}, the space $\Hom(\Gamma, G)$ can be arbitrarily singular. These properties reinforce the necessity to work in the setting of general complex analytic spaces.
\end{remark}

Given $\rho \in \A_{I},$ consider the Kodaira-Spencer map of the Anosov family over $\A_{I}$.
By a theorem of Goldman \cite{GOL84} (see also \cite{SIK12}), there is a natural isomorphism $T_{\rho}\A \simeq Z^{1}(\Gamma, \g_{\rho})$ where the latter is the complex vector space of group $1$-cocycles where the $\Gamma$-action on $\g$ is given by $\gamma\cdot X=\textnormal{Ad}(\rho(\gamma))(X).$
Therefore, the Kodaira-Spencer map takes the form of a linear map
\begin{align}
\KS_{\rho}: Z^{1}(\Gamma, \g_{\rho})\rightarrow H^{1}(\W_{\rho}^{I}, \Theta_{\W_{\rho}^{I}}).
\end{align}

Let $\g\rightarrow Z^{1}(\Gamma, \g_{\rho})$ be the boundary map in group cohomology.
Since the family $p: \W^{I}\rightarrow \A_I$ is $G$-equivariant, the image of $\g$ is always contained in the kernel of $\KS_{\rho}$.
The kernel and image of $\KS_{\rho}$ are studied in more detail in \autoref{sec:ks}.
However, the following example illustrates that the surjectivity of $\KS_{\rho}$ is a non-trivial matter.

\begin{example}\label{ex:GHYS}
Let $\rho: \Gamma \rightarrow \PSL(2, \C)$ be the inclusion of a torsion free cocompact lattice and $\iota_{3}: \PSL(2, \C)\rightarrow \PSL(3, \C)$ be the unique (up to conjugacy) irreducible homomorphism.  The homomorphism $\iota_{3}\circ \rho$ is $B$-Anosov where $B<\PSL(3, \C)$ is a Borel subgroup (see \cite{DS20}). 

Let $\F$ be the variety of complete flags in $\C^{3}$ and $I$ the unique balanced ideal of type $(B, \F).$  There is an isomorphism $\rho(\Gamma)\backslash \PSL(2, \C) \simeq \W_{\iota_{3}\circ \rho}^{I}$ (see \cite{ST15}).  By a theorem of Porti \cite{POR13}, the boundary map $\sl(3, \C)\rightarrow Z^{1}(\Gamma, \sl(3, \C)_{\iota_{3}\circ \rho})$ is an isomorphism.  Since $\sl(3, \C)$ is in the kernel of the Kodaira-Spencer map
\begin{align}
\KS_{\rho}: Z^{1}(\Gamma, \g_{\iota_{3}\circ \rho})\rightarrow H^{1}(\W_{\iota_{3}\circ \rho}^{I}, \Theta_{\W_{\iota_{3}\circ \rho}^{I}}),
\end{align}
we conclude that $\KS_{\rho}$ is zero.  Meanwhile, if the first Betti number of $\Gamma \backslash \PSL(2, \C)$ is positive, then Ghys \cite{GHY95} proved that the vector space $H^{1}(\W_{\iota_{3}\circ \rho}^{I}, \Theta_{\W_{\iota_{3}\circ \rho}^{I}})$ is nonzero, and hence the corresponding Anosov family is \emph{not} complete at $\iota_{3}\circ \rho.$
\end{example}

\subsection{The character scheme}
\label{subsec:charscheme}

Let $\X(\Gamma, G)$ denote the geometric invariant theory quotient of the affine scheme $\Hom(\Gamma, G)$ by the conjugation action of $G$.
Since $G$ is a reductive affine algebraic group, $\X(\Gamma, G)$ is an affine scheme of finite type over $\C$ and comes equipped with a map 
\begin{equation}
\label{eqn:chi-def}
\chi : \Hom(\Gamma, G)\rightarrow \X(\Gamma, G)
\end{equation}
that is constant on $G$-orbits (see Sikora \cite{SIK12}).  Via the analytification functor, we will simultaneously view $\X(\Gamma, G)$ as a complex analytic space.
However, it is important to note that the map $\Hom(\Gamma,G) \to \X(\Gamma,G)$ is not a set-theoretic quotient map; in general, its fibers may contain multiple $G$-orbits.

We record here the following \emph{quotient} theorem, which will be used to descend the the complex analytic family $\W^{I}$ over a suitable subset of $\A_{I}$ to the quotient by the $G$-action.
\begin{thm}\label{thm: quotients}
Let $X$ be a complex analytic space and $G$ a complex Lie group acting properly, freely, and holomorphically on $X$.
\begin{enumerate}
\item The quotient $X/G$ is a complex analytic space and the projection
        $\pi: X\rightarrow X/G$ is a smooth holomorphic map.
\item Let $\Y\rightarrow X$ be a $G$-equivariant complex analytic family.  Then      there exists a unique
        complex analytic family $\underline{\Y}\rightarrow X/G$ such that
        $\pi^{\star}\underline{\Y}\simeq \Y.$  
\end{enumerate}
\end{thm}
\begin{proof}
It was proved by Kaup \cite{KAU68} that the quotient $X/G$ is naturally a complex analytic space.  Let $x\in X$
and $\mathcal{O}_{x}$ the $G$-orbit of $x$.  Since the $G$-action is free and proper, the orbit map is a closed embedding yielding an isomorphism $G\simeq \mathcal{O}_{x}$. Moreover, the construction of the analytic structure on $X/G$ implies there exists an open neighborhood $U$ of $x,$ an open set $V$ of $\pi(x)$ such that $\pi(U)= V$, and a neighborhood $Q$ of the identity in $G$ and an isomorphism $V\times Q\simeq U$ commuting with the natural projections.  Hence, $\pi: X\rightarrow X/G$ is smooth.

Now, choose a sufficiently fine open cover $\{V_{i}\}$ of $X/G$ such that there exists a collection of open sets $\{U_{i}\}$ in $X$ such that $\pi(U_{i})=V_{i}$ and $U_{i}\simeq V_{i}\times Q_{i}$ as above.  Then, over each $V_{i}$ there exists a holomorphic section $\sigma_{i}: V_{i}\rightarrow X$ of $\pi.$  
Over $V_{i},$ define the complex analytic family $\uY_{i}:=\sigma_{i}^{\star}\Y.$

Now suppose $V_{ij}:=V_{i}\cap V_{j}\neq \emptyset.$  Since $G$ acts freely, there exists a unique holomorphic map $g_{ij}: V_{ij}\rightarrow G$ such that
$\sigma_{i}=g_{ij}\cdot \sigma_{j}.$  Furthermore, on triple intersections 
$g_{ik}\cdot \sigma_{k}=\sigma_{i}=g_{ij}\cdot \sigma_{j}=g_{ij}g_{jk}\cdot \sigma_{k}.$ Since $G$ acts freely, this implies $g_{ik}=g_{ij}g_{jk}.$  

Since $\Y$ is $G$-equivariant, the cocycle $\{g_{ij}\}$ defines isomorphisms
$\phi_{ij}: \uY_{i}\xrightarrow{\simeq} \uY_{j}.$  Moreover, the collection $\{\phi_{ij}\}$ also satisfy the cocycle condition.  Then, we define $\uY=\sqcup_{i} \uY_{i}/\sim$ where the equivalence relation $\sim$ is given by the cocycle $\{\phi_{ij}\}.$  Since the $G$-action is proper, it follows that the space $\uY$ is Hausdorff.  Therefore, by \cite[p.~20]{FIS76}, $\uY$ has a canonical complex analytic space structure.  Since smoothness is preserved by base change and $\uY$ is locally defined via $\sigma_{i}^{\star}\Y,$ the projection $\uY\rightarrow X/G$ is smooth and proper, therefore $\uY\rightarrow X/G.$ is a complex analytic family. The proof that $\pi^{\star}\uY\simeq \Y$ follows in the same fashion, which completes the proof.
\end{proof}

There are some mild additional hypotheses on homomorphisms that determine a large subset of $\Hom(\Gamma, G)$ where the hypotheses of \autoref{thm: quotients} are satisfied.
Specifically, following \cite{JM87}\cite{SIK12}, we say a homomorphism $\rho \in \Hom(\Gamma,G)$ is \emph{good} if it satisfies the conditions:
\begin{itemize}
\item It is irreducible, i.e.~$\rho(\Gamma)$ is not contained in a proper parabolic subgroup of $G$, and
\item The stabilizer $Z(\rho) = Z_G(\rho(\Gamma))$ of $\rho$ under the action of $G$ on $\Hom(\Gamma,G)$ is trivial.
\end{itemize}
The good homomorphisms $\Hom^{\vee}(\Gamma, G)$ form an open dense subset of the complex analytic space $\Hom(\Gamma,G)$ on which the action of $G$ is free and proper.  By \autoref{thm: quotients}, the quotient
$\X^{\vee}(\Gamma, G):=\Hom^{\vee}(\Gamma, G)/G$ is a complex analytic space.

Sometimes we will want to further restrict the homomorphisms we consider so that we work entirely in the category of complex manifolds.
We will say that $\rho\in \Hom(\Gamma, G)$ is \emph{very good} if it is good and $\rho$ is a smooth point of $\Hom(\Gamma,G)$.
The set of very good points $\Hom^{\star}(\Gamma,G)\subset\Hom(\Gamma,G)$ is a complex manifold on which $G$ acts freely and properly.
We denote the quotient complex manifold by $\X^{\star}(\Gamma, G)$.

\begin{prop}[{\cite[Thm.~53]{SIK12}}]
\label{prop:tangent}
For all $[\rho]\in \X^{\vee}(\Gamma, G)$ and any $\rho\in \Hom^{\vee}(\Gamma, G)$ projecting to $[\rho]$ (that is, such that $\chi(\rho) = [\rho]$, where $\chi$ is the map from \eqref{eqn:chi-def}), there is a natural isomorphism 
$T_{[\rho]}\X^{\vee}(\Gamma, G)\simeq H^1(\Gamma,\g_\rho)$.

Furthermore, the natural map $Z^1(\Gamma,\g_\rho) \to H^1(\Gamma,\g_\rho)$ represents the differential of $\chi$ in the sense that there is a commutative diagram:
\begin{center}
\begin{tikzcd}
T_\rho \Hom^{\vee}(\Gamma,G) \ar[r,"\sim"] \ar[d,"d_\rho \chi"] &  Z^1(\Gamma,\g_\rho) \ar[d] \\
T_{[\rho]} \X^{\vee}(\Gamma,G) \ar[r,"\sim"] &  H^1(\Gamma,\g_\rho)
\end{tikzcd}
\end{center}
\end{prop}

The following theorem of Guichard-Gu\'{e}ritaud-Kassel-Wienhard \cite[Proposition~1.8]{GGKW15} shows that the locus of Anosov representations is well defined in $\X(\Gamma, G).$  
\begin{thm}\label{Anosov GIT}
Let $\rho \in \Hom(\Gamma, G)$ and $\eta \in \overline{G\cdot \rho}$ be in the orbit closure of $\rho.$  Then $\rho\in \A$ if and only if $\eta\in \A.$  

Hence, there is a subset $\uA\subset \X(\Gamma, G)$ such that $\chi^{-1}(\uA)=\A.$  
\end{thm}

Now fix a balanced ideal $I$ of type $(P_{A}, \F)$.  Recall that $\A_{I} \subset \A$ denotes the subset of $\A$ where the domain of discontinuity corresponding to ideal $I$ is nonempty, and that we have the complex analytic family $\W^{I} \to \A_I$.
In general, we do not expect this analytic family over $\A_{I}$ to descend to a family over the quotient $\uA_{I} = \chi(\A_I) \subset \uA$, but as the statement of \autoref{thm: quotients} would suggest, this can be remedied by passing to the good (or very good) subset.

Specifically, let $\A_{I}^{\vee} := \A_I \cap \Hom^{\vee}(\Gamma, G)$ and $\A_I^\star := \A_I \cap \Hom^{\star}(\Gamma, G)$ denote the sets of good and very good representations in $\A_I$, so that $\A_I \supset \A_I^\vee \supset \A_I^\star$.
Let $\uA^{\vee}_I \subset \X^{\vee}(\Gamma, G)$ (a complex analytic space) and $\uA^{\star}_I$ (a complex manifold) denote the corresponding quotients.
Since $G$ acts properly and freely on $\A^{\vee}$, the following result is immediate from \autoref{thm: quotients}:
\begin{prop}
There is a unique complex analytic family $(\underline{\W}^{I}, \uA_{I}^{\vee}, \underline{p})$ such that $\pi^{\star}\underline{\W}^{I}\simeq \W^{I}$.
\end{prop}
Of course, when we want to work in the smooth setting we can further restrict the family given by this proposition to the very good locus, thus obtaining a complex analytic family over $\uA_I^\star$.  Generalizing our previous terminology, we refer to any of these families (i.e.~$\mathcal{W}^I$ or $\underline{\W}^I$ over any base space considered above) as an Anosov family.  The relations between the different homomorphism, character, and Anosov representation spaces discussed in this section are summarized in \autoref{fig:inclusions}.

\section{The Kodaira-Spencer map of the Anosov family}
\label{sec:ks}

\begin{figure}
  \[
    \begin{tikzcd}[remember picture]
    |[alias=col1top]|\Hom^\star(\Gamma,G) \ar[d,LA=1.25pt] \ar[r,symbol=\subset] &  \Hom^\vee(\Gamma,G) \ar[d,LA=1.25pt] \ar[r,symbol=\subset] & |[alias=septop]|\Hom(\Gamma,G) \ar[d,LA=0.5pt]   & |[alias=A11]|\A_I^\star \ar[d,LA=1.25pt] \ar[r,symbol=\subset] & |[alias=A12]|\A_I^\vee \ar[d,LA=1.25pt] \ar[r,symbol=\subset] & |[alias=A13]|\A_I  \ar[d,LA=0.5pt] \\
    |[alias=col1bot]|\X^\star(\Gamma,G) \ar[r,symbol=\subset] &  \X^\vee(\Gamma,G) \ar[r,symbol=\subset] & |[alias=sepbot]|\X(\Gamma,G)    & |[alias=A21]|\uA_I^\star \ar[r,symbol=\subset] & |[alias=A22]|\uA_I^\vee \ar[r,symbol=\subset] & \uA_I \\[-6mm]
    \text{\footnotesize \textbf{Very Good}} & \text{\footnotesize \textbf{Good}}\\[-9mm]
    \text{\footnotesize (manifolds)} & \text{\footnotesize (analytic spaces)} & \text{\footnotesize (affine schemes)}
    \end{tikzcd}
  \begin{tikzpicture}[overlay,remember picture]
  \draw[blue,dashed] ([yshift=2mm,xshift=16mm] septop.north) to ([yshift=-2mm,xshift=16mm] sepbot.south);
  \draw[red, rounded corners=5mm] ([yshift=2mm,xshift=-2mm]A11.north west) -- ([yshift=2mm,xshift=2mm] A13.north east) -- ([yshift=-2mm,xshift=2mm] A13.south east) -- ([xshift=2mm,
  yshift=-1.9mm]A12.south east) -- ([xshift=2mm,yshift=-2mm]A22.south east) -- ([xshift=-2mm,yshift=-2mm] A21.south west) -- cycle;
  \path (A11.north) -- node[above,yshift=2mm,red] {\footnotesize Can be base of an Anosov family} (A13.north);
  \end{tikzpicture}
  \]
  \vspace{-5mm}
  \caption{Relations between the spaces of homomorphisms and characters introduced in \autoref{subsec:charscheme} (left) and the corresponding sets of Anosov representations and characters (right).  Bold arrows are set-theoretic quotients by $G$ (each fiber is a single $G$-orbit), while regular arrows are GIT quotients (constant on orbits, but fibers can contain multiple $G$-orbits).}\label{fig:inclusions}
  \end{figure}
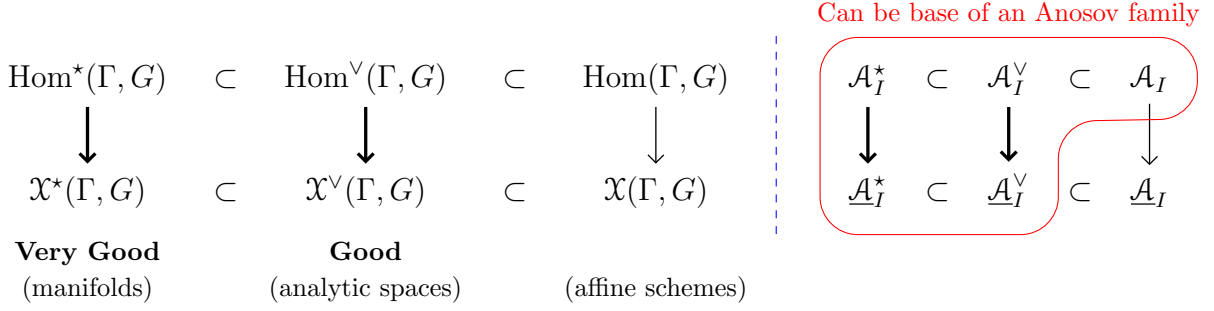

In this section, we develop the crucial technical tool which will allow us to compute the Kodaira-Spencer map of the Anosov family and prove \autoref{thm:main-infinitesimal}.

\begin{thm}
For any thickened Anosov homomorphism $(\rho, I)$ of type $(P_{A}, \F)$ such that $\rho\in \A_{I},$  there is a commutative diagram
\label{thm:pent}
\begin{equation}
\begin{tikzcd}
&  Z^1(\Gamma,\g_\rho) \ar[rd,"\KS_\rho"] \ar[d] \\
0 \ar[r] &  H^1(\Gamma,H^0(\Omega_{\rho}^{I}, \Theta_{\Omega_\rho^{I}})) \ar[r] & H^1(\W_{\rho}^{I},\Theta_{\W_\rho^{I}}) \ar[r] & H^0(\Gamma, H^1(\Omega_{\rho}^{I}, \Theta_{\Omega_\rho}^{I}))
\end{tikzcd}
\end{equation}
in which the bottom row is an exact sequence, and where the vertical arrow is the map on $1$-cocycles induced by the infinitesimal action of $G$ on $\F$ and restriction of vector fields from $\F$ to the open subset $\Omega_\rho^I$.
\end{thm}

In preparation for the proof we recall some constructions in homological algebra that we will need.

For a first-quadrant double complex $C = C^{\param,\param}$ with differentials of degrees $(1,0)$ and $(0,1)$, we denote by $\lv E_r^{p,q}(C)$ the associated vertical spectral sequence with differentials of degree $(r,1-r)$ and $\lv E_0^{p,q} = C^{p,q}$.
Similarly $\lh E_r^{p,q}(C)$ denotes the horizontal spectral sequence with differentials of degree $(1-r,r)$ and $\lh E_0^{p,q} = C^{p,q}$.

For a vertical spectral sequence $\lv E^{p,q}$ converging to $H^\param$ there is an exact sequence of low-degree terms
$$ 0 \to \lv E^{1,0}_2 \to H^1 \to \lv E^{0,1}_2 $$
which is functorial with respect to morphisms of spectral sequences.
For brevity we will refer to this as the LDT sequence.

Given a short exact sequence of sheaves,
$0 \to A \to B \to C \to 0$,
the connecting homomorphism $\delta^0 : H^0(C) \to H^1(A)$ from the associated long exact sequence of cohomology can be interpreted as the map $\lv E^{1,0}_2 \to H^1$ from the LDT sequence.
Specifically, suppose we have a vertical spectral sequence with $\lv E^1$ page
\begin{equation}
\lv E_1^{p,q} = \begin{cases}
H^q(B) & p=0\\
H^q(C) & p=1\\
0 & \text{otherwise}
\end{cases}
\label{eq:last-two-thirds}
\end{equation}
and differential induced by the map $B \to C$.
The sequence therefore converges to $H^\param(A)$ and has $\lv E_2^{1,0} = \coker \left (H^0(B) \to H^0(C) \right )$.
Furthermore, the initial terms of its LDT sequence give the injective quotient of $\delta^0$, i.e.~
\begin{equation}
\begin{tikzcd}
& H^0(C) \ar[d] \ar[rd,"\delta^0"] \\
0 \ar[r] & \lv E_2^{1,0} \ar[r] & H^1(A)
\end{tikzcd}
\label{eq:connecting-hom-conclusion}
\end{equation}
commutes, with the leftmost vertical arrow being the quotient by the image of $H^0(B)$.

Finally, let $M$ be a complex manifold and $p: \widetilde{M}\rightarrow M$
a regular cover with covering group $\Gamma$.  Let $\mathcal{L}$ be a locally free sheaf on $M$ and consider its pullback $p^{\star}\mathcal{L}$ to $\tilde{M}$.  This is a $\Gamma$-equivariant sheaf and therefore $H^{i}(\widetilde{M}, p^{\star}\mathcal{L})$ admits the structure of a $\Gamma$-module.  
Let $G$ be the global sections functor on $\widetilde{M}$ and $F$ the functor taking a $\Gamma$-module to the sub-module of $\Gamma$-invariants.  

Then, we have $(F\circ G)(p^{\star}\mathcal{L})=H^{0}(M, \mathcal{L}).$  The Grothendieck spectral sequence \cite{GRO57} can be applied to this composition to realize the space of sections (and indeed, all cohomology of $\mathcal{L}$) as the limit of a spectral sequence of a double complex $X$.
The construction of the double complex involves certain choices of resolutions, but using the Dolbeault resolution of $p^*\mathcal{L}$ and a chain complex $C^\param(\Gamma,-)$ computing group cohomology, it becomes
\begin{equation}
\label{eq:groth-double}
X^{p,q} = C^p(\Gamma,A^{0,q}(\Tilde{M},p^\star \L)).
\end{equation}
 Here $A^{p,q}(Y,\mathcal{E})$ denote the space of smooth $(p,q)$-forms on a complex manifold $Y$ with values in a holomorphic vector bundle $\mathcal{E}$.
 As with the cohomology of $p^\star \mathcal{L}$, the spaces $A^{0,q}(\Tilde{M},p^\star \L)$ have a natural $\Gamma$-module structure coming from the $\Gamma$-equivariant structure of $p^\star \mathcal{L}$.
 The construction of this spectral sequence is discussed in more detail in \autoref{app:homological}.

\begin{proof}[Proof of \autoref{thm:pent}.]
Let us apply the Grothendieck spectral sequence to the regular covering $\Omega_\rho^{I} \to \W_\rho^I$ and the equivariant vector bundle $\Theta_{\Omega_\rho^{I}}$.
The double complex of \eqref{eq:groth-double} becomes
$$X^{p,q} := C^p(\Gamma, A^{0,q}(\Omega_\rho^{I},\Theta_{\Omega_\rho^{I}}))$$
and we will take $C^p$ to be the group of inhomogeneous $p$-cochains for group cohomology.
The vertical differential of $X$ is $\dbar$ and the horizontal differential is $d_\Gamma$.
The limit of the vertical or horizontal spectral sequence associated to $X$ is the cohomology of the total complex, which by Grothendieck's theorem (explained in this instance in \autoref{thm:applied-groth} of \autoref{app:homological}) is canonically isomorphic to $H^\param(\W_\rho^{I},\Theta_{\Omega_\rho^{I}})$.

Considering the vertical spectral sequence of $X$, we find
$$\lv E_1^{p,q}(X) = C^p(\Gamma,H^q(\Omega_\rho^{I},\Theta_{\Omega_\rho^{I}})), \quad \lv E_2^{p,q}(X) = H^p(\Gamma,H^q(\Omega_\rho^{I},\Theta_{\Omega_\rho^{I}})),$$
and therefore the LDT sequence becomes
\begin{equation}
\label{eq:ltds1}
\begin{tikzcd}
0 \ar[r] &  H^1(\Gamma,H^0(\Omega_\rho^{I},\Theta_{\Omega_\rho^{I}})) \ar[r] & H^1(\W_\rho,\Theta_{\W_\rho^{I}}) \ar[r] & H^0(\Gamma, H^1(\Omega_\rho^{I},\Theta_{\Omega_\rho^{I}}))
\end{tikzcd}
\end{equation}
which is the lower row of \eqref{eq:pent-cd}.

To relate this construction to the Kodaira-Spencer map, we recall that $\KS_\rho$ is the connecting homomorphism of the exact sequence of sheaves
\begin{equation}
\label{eqn:ksexact}
0 \to \Theta_{\W_\rho^{I}} \to \iota_{\rho}^{*}\Theta_{\W^{I}} \to T_{\rho}\A_I
\otimes_{\C} \O_{\W_{\rho}^{I}} \to 0
\end{equation}
where $\iota_{\rho}: \W_{\rho}^{I}\rightarrow \W^{I}$ is the inclusion.

We can form another double complex $Y$ which will give rise to this map through the construction of \eqref{eq:last-two-thirds}--\eqref{eq:connecting-hom-conclusion}.
Define
$$Y^{p,q} = \begin{cases}
A^{0,q}(\W_\rho^{I},\iota_{\rho}^{*}\Theta_{\W^{I}}) & p=0\\
A^{0,q}(\W_\rho^{I},T_\rho \A_I \otimes_{\C} \O_{\W_{\rho}^{I}}) & p=1\\
0 & \text{otherwise}
\end{cases}
$$
where the vertical differential is $\dbar$, and where the nontrivial horizontal differential is induced by the vector bundle map $\iota_{\rho}^{*}\Theta_{\W^{I}} \to T_\rho \A_I \otimes_{\C}\O_{\W_{\rho}^{I}}$.
The total cohomology of this complex is easily seen to be $H^\param(\W_\rho^{I},\Theta_{\W_\rho^{I}})$:
The kernel of the horizontal differential is isomorphic to $A^{0,q}(\W_\rho^{I},\Theta_{\W_\rho^{I}})$, making the horizontal spectral sequence degenerate at the $\lh E_1(Y)$ page to the Dolbeault resolution of $\Theta_{\W_\rho^{I}}$.

Turning to the vertical spectral sequence of $Y$, we have
$$\lv E_1^{p,q}(Y) = \begin{cases}
H^q(\W_\rho^{I}, \iota_{\rho}^{*}\Theta_{\W^{I}}) & p=0\\
H^q(\W_\rho^{I}, T_\rho \A_I \otimes_{\C}\O_{\W_{\rho}^{I}}) & p=1\\
0 & \text{otherwise}
\end{cases},$$
which is \eqref{eq:last-two-thirds} specialized to this case.
By \eqref{eq:connecting-hom-conclusion}, the LDT sequence contains the connecting homomorphism, and we obtain
\begin{equation}
\begin{tikzcd}
& \hspace{-42mm}  H^0(\W_\rho^{I}, T_\rho \A_I \otimes_{\C}\O_{\W_{\rho}^{I}}) \simeq T_\rho\A_I \ar[d] \ar[rd,"\KS_\rho"] \\
0 \ar[r] & \lv E_2^{1,0}(Y) \ar[r] & H^1(\W_\rho^{I},\Theta_{\W_\rho^{I}})
\end{tikzcd}
\label{eq:ks-interpretation}
\end{equation}

Next, we claim that there exists a map of double complexes $F : Y \to X$ that induces an isomorphism $H^{1}(Y)\rightarrow H^{1}(X)$.
Before constructing the map, we note that its existence gives the diagram we seek.
By functoriality, such a map $F$ induces a commutative diagram mapping the LDT sequence of $\lv E(X)$ to that of $\lv E(Y)$ (which was determined in \eqref{eq:ltds1}), and in combination with \eqref{eq:ks-interpretation} we obtain:
\begin{equation}
\begin{tikzcd}
& \colorboxed{ T_\rho\A_I } \ar[d] \ar[rd,"\KS_\rho"] \\
0 \ar[r] & \lv E_2^{1,0}(Y) \ar[r] \ar[d,"F_*"] \ar[r] &  \colorboxed{ H^1(\W_\rho^{I},\Theta_{\W_\rho^{I}}) } \ar[d,"\sim"] \ar[r] & \lv E_2^{0,1}(Y) \ar[d,"F_*"] \\
\colorboxed{ 0 } \ar[r] &  \colorboxed{ H^1(\Gamma,H^0(\Omega_\rho^{I},\Theta_{\Omega_\rho^{I}})) } \ar[r] & H^1(\W_\rho^{I},\Theta_{\W_\rho^{I}}) \ar[r] & \colorboxed{ H^0(\Gamma, H^1(\Omega_\rho^{I},\Theta_{\Omega_\rho^{I}})) }
\end{tikzcd}
\label{eq:ltds-diagram}
\end{equation}
The induced diagram on the highlighted terms then becomes the desired one \eqref{eq:pent-cd} after using the isomorphism $T_\rho\A_I \simeq Z^1(\Gamma,\g_\rho)$.

The map of complexes $F$ will be constructed using the splitting of the tangent bundle of $\Omega^{I}$, which we now discuss.
Consider the short exact sequence
\begin{align}
0\rightarrow \Theta_{\Omega_{\rho}^{I}}\rightarrow \iota_{\rho}^{\star}\Theta_{\Omega^{I}} \rightarrow T_{\rho}\A_I \otimes \mathcal{O}_{\Omega_{\rho}^{I}}\rightarrow 0
\end{align}
where $\iota_{\rho}: \Omega_{\rho}^{I}\rightarrow \Omega^{I}$ is the inclusion.
Since $\Omega^{I}\subset \F\times \A_{I}$ is an open subset of a product, the bundle $\iota_{\rho}^{\star}\Theta_{\Omega^{I}}$ admits a canonical holomorphic splitting $\iota_{\rho}^{\star}\Theta_{\Omega^{I}}\simeq \mathcal{H}\oplus \mathcal{V}$ 
where $\mathcal{H} =T_{\rho}\A_I\otimes \mathcal{O}_{\Omega_{\rho}^{I}}$
and $\mathcal{V}=\Theta_{\Omega_{\rho}^{I}}.$
For a local section $\sigma$ of $\iota_{\rho}^{\star} \Theta_{\Omega^{I}}$, or for a form with values in this bundle, we denote by $\sigma^{\mathcal{H}}$ and $\sigma^{\mathcal{V}}$ its components with respect to the splitting.

The action of $\Gamma$ on $\F \times \A_{I}$ by $\gamma \cdot (x,\rho) := (\rho(\gamma)\cdot x, \rho)$ preserves the vertical slice $\iota_\rho(\Omega_{\rho}^{I})$ and hence both $\Theta_{\Omega^{I}}$ and $\iota_{\rho}^{\star} \Theta_{\Omega}^{I}$ have $\Gamma$-equivariant structures.
While this action preserves the vertical distribution $\mathcal{V}$, it does not preserve $\mathcal{H}$.
Indeed, the failure of invariance of $\mathcal{H}$ precisely captures the infinitesimal action of a deformation of $\rho$ on $\F$ as follows:
Applying the differential of the action $\rho(\gamma)\cdot x$ to a horizontal vector $(0, \dot{\rho}) \in T_{(x,\rho)}(\F \times \A_{I})$ gives $(\dot{\rho}(\gamma)^\sharp(y), \dot{\rho})\in T_{(y,\rho)}(\F \times \A_I)$, where $y=\rho(\gamma)\cdot x$ and $v^\sharp \in H^0(\F, \Theta_{\F})$ denotes the action vector field corresponding to Lie algebra element $v \in \g$.

Now we define the map $F$.
Since the complex $Y$ has only two nontrivial columns, we specify its action on each column separately.
We begin with $F^{0,q}$.
Since $Y^{0,q} = A^{0,q}(\W_\rho^{I},\iota_{\rho}^{\star} \Theta_{\W^{I}})$ and $X^{0,q}  = C^0(\Gamma,A^{0,q}(\Omega_\rho^{I},\Theta_{\Omega_\rho^{I}})) = A^{0,q}(\Omega_\rho^{I},\Theta_{\Omega_\rho^{I}})$, we seek a map
$$ A^{0,q}(\W_\rho^{I}, \iota_{\rho}^{\star} \Theta_{\W^{I}}) \to A^{0,q}(\Omega_\rho^{I},\Theta_{\Omega_\rho^{I}}).$$
Given a form $\alpha\in A^{0,q}(\W_\rho^{I},\iota_{\rho}^{\star} \Theta_{\W^{I}})$, let $\tilde{\alpha}$ denote its pullback to a form on $\Omega_\rho^{I}$ with values in $\iota_{\rho}^{\star} \Theta_{\Omega^{I}}$.
Any such lifted form is, of course, $\Gamma$-invariant.
Then we define 
\begin{equation}
F^{0,q}(\alpha) = \tilde{\alpha}^{\mathcal{V}}.
\end{equation}

Turning to $F^{1,q}$, we must define a map
\[
A^{0,q}(\W_{\rho}^{I}, T_\rho \A_I\otimes_{\C} \O_{\W_{\rho}^{I}}) \to  C^1(\Gamma,A^{0,q}(\Omega_\rho^{I},\Theta_{\Omega_\rho^{I}}))
\]
Now let $\beta \in A^{0,q}(\W_{\rho}^{I}, T_\rho \A_I\otimes_{\C} \O_{\W_{\rho}^{I}}).$  The pullback of $\beta$ to $\Omega_{\rho}^{I}$ can then be regarded as a form with values in $\mathcal{H}$ using the splitting discussed above.
Let us denote by $\hat{\beta}$ the resulting horizontal $\iota_{\rho}^{\star}\Theta_{\Omega^{I}}$-valued form.
The form $\hat{\beta}$ is typically not $\Gamma$-invariant, as the horizontal distribution is not preserved by the action of $\Gamma$. 
The failure of $\Gamma$-invariance is naturally encoded in a $1$-cocycle for $\Gamma$, which leads to the definition of $F^{1,q}$:
\begin{equation}
F^{1,q}(\beta)(\gamma) = (\hat{\beta} - \gamma\cdot \hat{\beta})^{\mathcal{V}} = - ( \gamma\cdot \hat{\beta} )^{\mathcal{V}}
\end{equation}
where $\gamma \cdot \hat{\beta}$ is the induced $\Gamma$-action on $A^{0,q}(\Omega_{\rho}, \iota_{\rho}^{\star}\Theta_{\Omega_{\rho}^{I}}).$
Since the vertical projection is holomorphic, the map $F$ commutes with $\dbar$, and hence is a map of vertical complexes.
To check that it is a map of double complexes, we must verify that $d_\Gamma(F^{0,q}(\alpha)) = F^{1,q}(p_*\alpha)$ where
$p: \iota_{\rho}^{\star}\Theta_{\Omega^{I}}\rightarrow T_\rho \A_I\otimes_{\C} \O_{\W_{\rho}^{I}}.$
It is immediate from the definitions that $\hat{p_*\alpha} = \tilde{\alpha}^{\mathcal{H}}$, and thus for $\gamma \in \Gamma$ we have
\begin{equation*}
\begin{split}
d_\Gamma(F^{0,q}(\alpha)) (\gamma) &= \gamma\cdot(\tilde{\alpha}^{\mathcal{V}}) - \tilde{\alpha}^{\mathcal{V}}\\
F^{1,q}(p_*\alpha) (\gamma) &= - (\gamma\cdot (\tilde{\alpha}^{\mathcal{H}}))^{\mathcal{V}}
\end{split}
\end{equation*}
Recalling that the vertical distribution is $\Gamma$-invariant, we have $\gamma\cdot(\tilde{\alpha}^{\mathcal{V}}) = \left ( \gamma\cdot(\tilde{\alpha}^{\mathcal{V}}) \right )^{\mathcal{V}}$; using this, and the splitting $\tilde{\alpha} = \tilde{\alpha}^{\mathcal{H}} + \tilde{\alpha}^{\mathcal{V}}$, we calculate
\begin{equation*}
\begin{split}
d_\Gamma(F^{0,q}(\alpha)) (\gamma) - F^{1,q}(p_*\alpha) (\gamma) &= \left (\gamma\cdot(\tilde{\alpha}^{\mathcal{V}}) \right)^{\mathcal{V}} - \tilde{\alpha}^{\mathcal{V}} + \left (\gamma\cdot (\tilde{\alpha}^{\mathcal{H}}) \right )^{\mathcal{V}}\\
&= \left ( \gamma\cdot(\tilde{\alpha}^{\mathcal{H}} + \tilde{\alpha}^{\mathcal{V}}) - \tilde{\alpha} \right )^{\mathcal{V}}\\
&= (\gamma\cdot \tilde{\alpha} - \tilde{\alpha})^{\mathcal{V}} = 0.
\end{split}
\end{equation*}
where vanishing of the last quantity follows because $\tilde{\alpha}$ is $\Gamma$-invariant.
We conclude $F$ is a map of double complexes.

We now show that $F$ induces an isomorphism on $H^{1}.$  Observe that
$$ \lh E_1^{p,q}(X) \simeq \begin{cases} A^{0,q}(\W_\rho^{I},\Theta_{\W_\rho^{I}}) & p=0\\ H^{p}(\Gamma,A^{0,q}(\W_\rho^{I},\Theta_{\W_\rho^{I}})) & p>0\end{cases},$$
and
$$ \lh E_1^{p,q}(Y) \simeq \begin{cases} A^{0,q}(\W_\rho^{I},\Theta_{\W_\rho^{I}}) & p=0\\ 0 & p>0\end{cases}.$$
Chasing the definition of $F,$ it is straightforward to see that the induced map between spectral sequences $F: \lh E_1^{p,q}(Y)\rightarrow E_1^{p,q}(X)$ is the identity for $p=0$ and zero elsewhere.  Upon passing to the $E_{2}$-page, and remembering that each spectral sequence converges to $H^{\param}(\W_\rho^{I},\Theta_{\W_\rho^{I}}),$ it quickly follows that $F$ is an isomorphism (in fact the identity) on $H^{1}.$  

Finally, it remains to check the description of the vertical map in diagram \eqref{eq:pent-cd}. 
Since this map is constructed from the leftmost column of vertical maps in \eqref{eq:ltds-diagram}, we see that to characterize it, we must first map an arbitrary element $\dot{\rho} \in Z^1(\Gamma,\g_\rho) \simeq T_\rho \A_I$ to
$\lv E_2^{1,0}(Y) = \coker\left (A^{0,0}(\W_\rho^{I},\iota_{\rho}^{*}\Theta_{\W^{I}}) \to A^{0,0}(\W_\rho^{I},T_\rho \A_I\otimes_{\C} \O_{\W_{\rho}^{I}}) \right )$.
Of course $A^{0,0}(\W_\rho^{I},T_\rho \A_I\otimes_{\C} \O_{\W_{\rho}^{I}})$ is simply the space $C^\infty(\W_\rho^{I},T_\rho\mathcal{A})$ of smooth $T_\rho\mathcal{A}$-valued functions.
The image of $\dot{\rho}$ then corresponds to the equivalence class of the \emph{constant} function $\W_\rho^I \to T_\rho\mathcal{A}$ with value $\dot{\rho}$.

Next we must apply the lower left vertical map from \eqref{eq:ltds-diagram}, which is induced by $F^{1,0}$.
As observed in the discussion of the horizontal distribution, pushing the horizontal section over $\Omega_\rho^{I}$ corresponding to $\dot{\rho}$ forward by an element $\gamma \in \Gamma$ gives a section whose vertical component is the vector field $\dot{\rho}(\gamma)^\sharp$ on $\Omega_{\rho}^{I}$.
Since $F^{1,0}$ is defined by taking this vertical component, we find that the image of $\dot{\rho}$ under the vertical map of \eqref{eq:pent-cd} is the element of $H^1(\Gamma,H^0(\Omega_\rho^i, \Theta_{\Omega_\rho^i}))$ corresponding to the $1$-cochain $\gamma \mapsto \left . \dot{\rho}(\gamma)^\sharp \right |_{\Omega^I_\rho}$, as claimed.
\end{proof}

\section{Small limit sets and the Anosov family}
\label{sec:small}

Building on \autoref{thm:pent}, we will now explore the deformation-theoretic consequences of small Hausdorff dimension for the limit set of an Anosov representation.
As before we suppose $(\rho,I)$ is a thickened Anosov homomorphism of type $(P_A,\F)$ as defined in \autoref{subsec:domains}.
Recall from \autoref{subsec:hdim} that such $(\rho,I)$ is said to be $k$-small if its limit set $\Lambda_\rho^I \subset \F$ is a null set for the Hausdorff measure $\mathcal{H}_{2N-k}$ of dimension $2N-k$, where $N = \dim_\C\F$.
In general, whenever we refer to the $d$-dimensional Hausdorff measure $\mathcal{H}_d$ on a manifold, it is assumed to be the one associated to the distance function of some Riemannian metric.
Our statements involving such measures will be true regardless of which Riemannian metric is used.

Let $\z(\rho)$ denote the Lie algebra of the centralizer $Z(\rho) := Z_G(\rho(\Gamma)$.  Since $\g\simeq H^{0}(\F, \Theta_{\F})$ consists of holomorphic vector fields on $\F$, restriction to the open $\rho(\Gamma)$-invariant set $\Omega_\rho^{I}$ gives a natural injective map of $\Gamma$-modules
\begin{equation}
\label{eq:h0Omega}
\g_\rho \to H^0(\Omega_\rho^{I}, \Theta_{\Omega_\rho^{I}}).
\end{equation}
Similarly, since $\z(\rho)$ consists of holomorphic vector fields on $\F$ invariant under $\rho(\Gamma)$, restricting these to $\Omega_\rho^{I}$ and taking the quotient gives a natural injective linear map
\begin{equation}
\label{eq:h0W}
\z(\rho) \to H^0(\W_\rho^{I}, \Theta_{\W_\rho^{I}}),
\end{equation}
which is just the $\Gamma$-invariant part of \eqref{eq:h0Omega}.

The following result of Harvey is the core principle that we use to draw complex-analytic conclusions from a $k$-smallness hypothesis:

\begin{thm}[{\cite[Theorems 1 and 4]{HAR74}}]\label{thm:harvey}
Let $Y$ be a complex manifold of dimension $n$, and let $m$ be a nonnegative integer.
If $E \subset Y$ is a closed subset satisfying $\mathcal{H}_{2n-2m-2}(E)=0$, then every locally free sheaf $\mathcal{L}$ on $Y - E$ is the restriction of a unique locally free sheaf on $Y$ (which we also denote by $\mathcal{L}$), and the inclusion $(Y - E) \into Y$ induces an isomorphism
\[ H^i(Y,\mathcal{L}) \to H^i(Y -E, \mathcal{L}) \]
for all $0 \leq i \leq m$.
\end{thm}

Applying this to our situation of interest, we find:

\begin{thm}\mbox{}
\label{thm:h0}
Let $(\rho, I)$ be a thickened Anosov homomorphism of type $(P_{A}, \F).$
\begin{rmenumerate}
\item If $(\rho, I)$ is $2$-small, then the maps $\g_\rho \to H^0(\Omega_\rho^{I}, \Theta_{\Omega_\rho^{I}})$ and $\z(\rho) \to H^0(\W_\rho^{I}, \Theta_{\W_\rho^{I}})$ are isomorphisms.
\item If $(\rho, I)$ is $2$-small and good (as defined in \autoref{subsec:charscheme}), then $H^0(\W_\rho^{I},\Theta_{\W_\rho^{I}}) = 0$.
\item If $(\rho, I)$ is $4$-small, then $H^1(\Omega_\rho^{I},\Theta_{\Omega_\rho^{I}}) = 0$.
\end{rmenumerate}
\end{thm}

\begin{proof}
First we observe that if $\rho$ is $k$-small for $k\geq0$, then the associated domain in $\mathcal{F}$ is nonempty and hence $\rho \in \A_I$.

If $\rho$ is $(2m+2)$-small, then the hypotheses of \autoref{thm:harvey} are satisfied for the complex manifold $\F$, the closed subset $\Lambda_\rho^{I}$, and any locally free sheaf $\mathcal{L}$.
We will apply the theorem in this way several times.

For statement (i), note that $\g = H^0(\F,\Theta_\F)$ and take $m=0$ and $\mathcal{L} = \Theta_\F$.
The conclusion of \autoref{thm:harvey} is this case is that the $\Gamma$-module map $\g_\rho \simeq H^0(\Omega_\rho^{I}, \Theta_{\Omega_\rho^{I}})$ is an isomorphism (of vector spaces, and hence of $\Gamma$-modules) for any $2$-small thickened Anosov homomorphism $(\rho,I)$.
The associated isomorphism of $\Gamma$-invariant subspaces then gives
\[ 
\z(\rho) \simeq H^0(\Omega_\rho^{I}, \Theta_{\Omega_\rho^{I}})^\Gamma \simeq H^0(\W_\rho^{I}, \Theta_{\W_\rho^{I}}),
\]
completing the proof of (i).

Statement (ii) is an immediate consequence of (i) and the fact that $\z(\rho)=0$ for good representations.

Finally, if we proceed as in (i) but take $m=1$, we find that for $4$-small representations, $H^1(\Omega_\rho^{I},\Theta_{\Omega_\rho^{I}}) \simeq H^1(\F,\Theta_\F)$.
By a theorem of Bott \cite{BOT57}, $H^1(\F,\Theta_\F)=\{0\}$ and we obtain (iii).
\end{proof}

Using the main result of \autoref{thm:pent} we can give a criterion for completeness of the Anosov family:

\begin{thm}
\label{thm:completeness}
Let $(\rho, I)$ be a thickened Anosov homomorphism of type $(P_{A}, \F)$ such that $(\rho, I)$ is $4$-small.  Then the Kodaira-Spencer map of the Anosov family at $\rho$ factors as the natural surjection $Z^1(\Gamma,\g_\rho) \to H^1(\Gamma,\g_\rho)$ composed with an isomorphism $H^1(\Gamma,\g_\rho) \simeq H^1(\W_\rho^{I}, \Theta_{\W_\rho^{I}})$.  In particular the Kodaira-Spencer map is surjective, with kernel equal to $B^1(\Gamma,\g_\rho)$.
\end{thm}

\begin{proof}
We consider the commutative diagram of \autoref{thm:pent}.
By \autoref{thm:h0}, we have $H^0(\Gamma,H^1(\Omega_{\rho}^{I},\Theta_{\Omega_\rho^{I}}))=0$ and $H^1(\Gamma,H^0(\Omega_{\rho}^{I}, \Theta_{\Omega_\rho^{I}})) \simeq H^1(\Gamma,\g_\rho)$.
Thus the diagram becomes:
\begin{equation}
\begin{tikzcd}
&  Z^1(\Gamma,\g_\rho) \ar[rd,"\KS_\rho"] \ar[d] \\
0 \ar[r] &  H^1(\Gamma,\g_\rho) \ar[r,"\sim"] & H^1(\W_{\rho}^{I},\Theta_{\W_\rho^{I}}) \ar[r] & 0
\end{tikzcd}
\label{eq:pent-cd}
\end{equation}
Finally, by the description of the vertical map given in \autoref{thm:pent}, the vertical map in the diagram above is the natural surjection $Z^1(\Gamma,\g_\rho) \to H^1(\Gamma,\g_\rho)$.
Hence we have obtained the desired factorization of $\KS_\rho$.
\end{proof}

At this point we have completed the proof of \autoref{thm:main-infinitesimal}: Parts (i) and (ii) are contained in \autoref{thm:completeness}, while part (iii) is \autoref{thm:h0}(i).

\autoref{thm:completeness} immediately implies the following rigidity comparison result, which was stated in the introduction as \autoref{cor:main-infinitesimal}:

\begin{cor}\label{cor: rigid}
Let $(\rho, I)$ be a thickened Anosov homomorphism of type $(P_{A}, \F)$ such that $(\rho, I)$ is $4$-small. Then the complex manifold $\W_{\rho}^{I}$ is infinitesimally rigid if and only if the homomorphism $\rho: \Gamma\rightarrow G$ is infinitesimally rigid modulo conjugation.
\end{cor}

\begin{proof}
By \autoref{thm:completeness}, $H^{1}(\Gamma, \g_{\rho})\simeq H^{1}(\W_{\rho}^{I}, \Theta_{\rho}^{I})$ and therefore $H^{1}(\Gamma, \g_{\rho})=\{0\}$ if and only if $H^{1}(\W_{\rho}^{I}, \Theta_{\rho}^{I})=\{0\}.$  This completes the proof.
\end{proof}
\autoref{cor: rigid} gives a method to construct many new examples of infinitesimally rigid complex manifolds.
\begin{example} \label{ex:rigid}
Let $\rho: \Gamma\rightarrow \PSL(2, \C)$ be the inclusion of a torsion free cocompact lattice.  Let $\iota_{n}: \PSL(2, \C)\rightarrow \PSL(n, \C)$ denote the irreducible representation where $n\geq 4.$  By a theorem of Porti \cite{POR13}, $H^{1}(\Gamma, \mathfrak{sl}(n, \C)_{\iota_{n}\circ \rho})=\{0\}.$  The homomorphism $\iota_{n}\circ \rho$ is Anosov with respect to a Borel subgroup of $\PSL(n,\C)$ (see \cite{DS20}) and therefore is $P_{A}$-Anosov for any symmetric parabolic subgroup $P_{A}<\PSL(n, \C).$

Let $\F$ be the variety of complete flags in $\C^{n}.$
For large enough $n$, there always exists balanced ideals $I$ such that the thickened Anosov homomorphism $(\iota_{n}\circ \rho, I)$ of type $(P_{A}, \F)$ is $4$-small (see \cite{DS20}).  Therefore, by \autoref{cor: rigid} we obtain $H^{1}(\W_{\iota_{n}\circ \rho}^{I}, \Theta_{\W_{\iota_{n}\circ \rho}^{I}})=\{0\}.$  Hence, $\W_{\iota_{n}\circ \rho}^{I}$ is infinitesimally rigid.

By \cite[Theorems D, E, and 6.10]{DS20}, each rigid manifold obtained this way has the following properties:
\begin{rmenumerate}
  \item It is not K\"ahler (and hence is not a projective variety),
  \item The Kodaira dimension is $-\infty$, 
  \item The fundamental group is infinite, and
  \item The universal cover does not admit nonconstant holomorphic functions.
\end{rmenumerate}
These properties distinguish our examples from the rigid compact complex manifolds of higher dimension ($\geq\!3$) previously described in the literature, e.g.~in \cite{BOT57}\cite{KS58}\cite{CalVes59}\cite{Rag66}\cite{BC18}\cite{IG21}.

It \emph{is} necessary to exclude small $n$ from this construction to obtain rigidity.
For $n=3$ there is a unique balanced ideal $I$ and $(\rho, I)$ is \emph{not} $4$-small.
In \autoref{ex:GHYS} we noted that the corresponding manifold in this case is often not infinitesimally rigid (as shown in \cite{GHY95}).

We now discuss the geometry of the manifold $\W_{\iota_{n}\circ \rho}^{I}$ a bit more.  By a theorem of Sepp\"{a}nen-Tsanov \cite{ST15}, for every very ample holomorphic $\PSL(2, \C)$-equivariant line bundle $\mathcal{L}$ on $\F,$
there is a corresponding balanced ideal of type $(B, \F)$ and a holomorphic map
\begin{align}\label{alg reduction}
\W_{\iota_{n}\circ \rho}^{I}\rightarrow \F/\!/_{\mathcal{L}} \PSL(2, \C)
\end{align}
with fibers isomorphic to $\Gamma\backslash \PSL(2, \C)$ where
$\F/\!/_{\mathcal{L}} \PSL(2, \C)$ is the projective GIT quotient (polarized by $\mathcal{L}$).  In fact, it is shown in \cite{ST15} that $\F/\!/_{\mathcal{L}} \PSL(2, \C)\simeq \PSL(2, \C)\backslash \Omega_{\iota_{n}\circ \rho}^{I}.$

As noted above, the manifold $\W_{\iota_{n}\circ \rho}^{I}$ is not a K\"{a}hler manifold, therefore it is not a projective variety.
Somewhat informally, this construction shows that the manifold $\W_{\iota_{n}\circ \rho}^{I}$ becomes a complex projective variety after contracting the fibers of the map \eqref{alg reduction}, which are themselves isomorphic to the (often non-rigid) non-K\"ahler examples of Ghys from \autoref{ex:GHYS}.
\end{example}

\medskip

We now apply \autoref{thm:completeness} to local Anosov families over the character variety.  Recall that the complex manifold $\A^\star_I$ of very good Anosov representations having nonempty $\Omega_\rho^I$ was defined in \autoref{subsec:charscheme} (and its relation to other spaces of representations is summarized in \autoref{fig:inclusions}).  The statements that follow are formulated for subsets of this manifold, rather than the larger complex analytic space $\A^\vee_I$, due to the smoothness hypothesis in \autoref{thm:completeness}.

\begin{thm}
\label{cor:universality}
Let $I$ be a balanced ideal of type $(P_{A}, \F).$  
Let $V \subset \A^{\star}_I$ be a $G$-invariant open set such that for every $\rho\in V,$ the thickened Anosov homomorphism $(\rho, I)$ of type $(P_{A}, \F)$ is $4$-small.  Let $\underline{V}\subset \uA^{\star}_I$ be the corresponding quotient.
Then the Anosov family $\underline{\W}^{I}$ over $V$ is universal.
\end{thm}

\begin{proof}
Recall that the Anosov family over $\underline{V}$ is locally obtained by the pullback of $\W^{I}$ by a holomorphic section $\sigma : \underline{V} \to \A^{\star}_I$.
Let $[\rho] \in \underline{V}$ and suppose $\sigma([\rho]) = \rho$.
Let $\KS_{[\rho]} : T_{[\rho]} \underline{V} \to H^1(\W_\rho,\Theta_{\W_\rho})$ be the Kodaira-Spencer map of the family $\underline{\W}^{I}$ over $\underline{V}$.
Functoriality of the Kodaira-Spencer map with respect to pullback of families then gives a factorization of $\KS_{[\rho]}$:
\[
\begin{tikzcd}
T_{[\rho]} \underline{V} \ar[rd,"d_{[\rho]} \sigma"] \ar[rr, "\KS_{[\rho]}"] & & H^1(\W_\rho^{I},\Theta_{\W_\rho^{I}})\\
& T_{\rho} V \ar[ru,"\KS_\rho"]
\end{tikzcd}
\]
By the description of tangent spaces from \autoref{prop:tangent},  This diagram is isomorphic to the diagram
\begin{center}
\begin{tikzcd}
H^{1}(\Gamma, \g_{\rho}) \ar[rd,"d_{[\rho]} \sigma"] \ar[rr, "\KS_{[\rho]}"] & & H^1(\W_\rho^{I},\Theta_{\W_\rho^{I}})\\
& Z^{1}(\Gamma, \g_{\rho}) \ar[ru,"\KS_\rho"]
\end{tikzcd}
\end{center}
Since $(\rho, I)$ is $4$-small, \autoref{thm:completeness} implies that $\KS_{\rho}$ is surjective, therefore $\KS_{[\rho]}$ is surjective.

Since $\sigma$ is a section, $d_{[\rho]}\sigma: H^{1}(\Gamma, \g_{\rho})\rightarrow Z^{1}(\Gamma, \g_{\rho})$ is injective and split by the natural surjection $Z^{1}(\Gamma, \g_{\rho})\rightarrow H^{1}(\Gamma, \g_{\rho}).$ Hence, the commutativity of the above triangle implies that $\KS_{[\rho]}$ is injective.  Hence, $\KS_{[\rho]}$ is an isomorphism and the family $\underline{\W}^{I}$ over $\underline{V}$ is complete and effective at every point of $\underline{V}.$  

Moreover, since every $\rho\in V$ has trivial centralizer, $H^{0}(\W_{\rho}^{I}, \Theta_{\W_{\rho}^{I}})\simeq H^{0}(\Gamma, \g_{\rho})=\{0\}.$
Hence we can apply \autoref{prop:universality-criterion} at each point to conclude that the family $\underline{\W}^{I}$ over $\underline{V}$ is universal.
\end{proof}

We now state a general result for closed surface groups.
\begin{cor}
\label{cor:surface-completeness}
Suppose $\Gamma = \pi_1(S)$ is a surface group.  Suppose $G$ is a complex simple Lie group not of type $A_1$, $A_2$, $A_3$, or $B_2$. Let $I$ be a balanced ideal of type $(P_{A}, \F)$ such that that $G$-quasi-Fuchsian and $G$-Hitchin representations give rise to nonempty domains $\Omega_\rho^I$.

Then there exists a connected non-empty open $G$-invariant set $\mathcal{U}\subset \A_{I}^{\star}$ satisfying the following:
\begin{itemize}
\item  The open set $\mathcal{U}$ contains the set of $G$-quasi-Fuchsian homomorphisms.  If $G$ is not of type $F_{4}, E_{6}, E_{7}$ or $E_{8},$ then $\mathcal{U}$ also contains the set of $G$-Hitchin homomorphisms.
\item  Let $\underline{\mathcal{U}}\subset \uA_{I}^{\star}$ denote the corresponding quotient.  Then the restriction of the family $\underline{\W}^{I}$ to $\underline{\mathcal{U}}$ is universal.
\end{itemize}
\end{cor}

Thus we conclude that for simple $G$ of sufficiently high rank, any homomorphism sufficiently close to a $G$-quasi-Fuchsian or $G$-Hitchin homomorphism has a neighborhood in the character variety over which the Anosov family is universal.

\begin{proof}
Let $\mathcal{U}_0\subset \A$ be the open $G$-invariant subset provided by \autoref{thm: PSW}.  Then for all $\rho\in \mathcal{U}_0$ and any balanced ideal $I$ of type $(P_{A}, \F),$ the thickened Anosov homomorphism $(\rho, I)$ is $4$-small.

Let $\mathcal{U} = \mathcal{U}_0 \cap \A_I^\star$, which is again a $G$-invariant open set.
It still contains the relevant class of homomorphisms ($G$-quasi-Fuchsian or $G$-Hitchin), since by
\autoref{prop:GQF} and \autoref{prop:GHIT} these types of homomorphisms are smooth points, and we assumed that they give nonempty domains of discontinuity.  Therefore, we may apply \autoref{cor:universality} to conclude that the restriction of $\underline{\W}^{I}$ to $\underline{\mathcal{U}}$ is universal.  This completes the proof.
\end{proof}

\section{Teichm\"uller space}
\label{sec:teich}

We refer to the work of Catanese \cite{CAT11} and Meersseman \cite{MEE19} for details on Teichm\"{u}ller spaces of higher-dimensional manifolds.

Let $M$ be an oriented smooth closed manifold of even dimension.
Let $\mathcal{C}(M)$ denote the space of all smooth, integrable almost complex structures on $M$ compatible with its orientation.  The space $\mathcal{C}(M)$ is a subset of the linear Fr\'{e}chet space of sections $A^{0}(M, \textnormal{End}(TM)),$ and therefore inherits a topology.  The Fr\'{e}chet Lie group $\textnormal{Diff}_{0}(M)$ of diffeomorphisms isotopic to the identity acts continuously on $\mathcal{C}(M)$ by pullback.
\begin{defn}
The Teichm\"{u}ller space of $M$ is the topological space
\begin{align}
\mathcal{T}_{M}:=\mathcal{C}(M)/\textnormal{Diff}_{0}(M)
\end{align}
equipped with the quotient topology.
\end{defn}
Note that if $M$ is itself a complex manifold, then the associated almost complex structure gives $\T(M)$ a natural base point, which we denote by $[M]$. 

\begin{remark}
In contrast to the case of Riemann surfaces, in general the Teichm\"{u}ller space of a higher-dimensional manifold can be very pathological.  In particular, it may not be locally Hausdorff, and therefore may not admit a local complex analytic space structure.
Below, we will give some interesting examples of open subsets of $\mathcal{T}(M)$ which admit the structure of a complex manifold.  The question of how to view $\mathcal{T}(M)$ as a global complex analytic object was recently answered by Meersseman \cite{MEE19}: it is a certain limit of complex analytic Artin stacks.  As a word of warning, in contrast to the algebraic setting, there is not a uniformly accepted definition of complex analytic Artin stack, and the reader should consult \cite{MEE19} for the precise meaning in this setting.
\end{remark}

Let $p : \Y \to B$ be a family of complex manifolds with $\Y$ and $B$ smooth and connected which is locally trivial in the $C^{\infty}$-topology.  Then, given any $b\in B,$ the family $p: \Y\rightarrow B$ admits the structure of a locally trivial fiber bundle (in the $C^{\infty}$-category) with typical fiber $\Y_{b}$ and structure group $\textnormal{Diff}(\Y_{b}).$  

A \emph{marking} $\mu$ of $p:\Y \to B$ is a reduction of structure group to $\Diff_0(\Y_{b})$.
In other words, a marking is given by a compatible atlas of local trivializations $p^{-1}(U_\alpha) \simeq U_\alpha \times \Y_{b}$ whose transition maps take values in $\Diff_0(\Y_{b})$.
In particular, a marking yields a preferred isotopy class of diffeomorphism
between $\Y_{b'}$ and $\Y_{b}$ for all $b'\in B.$  

Consider the exact sequence
\begin{align}
1\rightarrow \Diff_0(\Y_{b})\rightarrow \Diff(\Y_{b})\rightarrow \Mod(\Y_{b})\rightarrow 1
\end{align}
where $\Mod(\Y_{b})$ is the mapping class group of isotopy classes of diffeomorphisms.  There exists a monodromy homomorphism $\mu_{p}: \pi_{1}(B,b)\rightarrow \Mod(\Y_{b}).$  
The existence of a marking for the family $p: \Y\rightarrow B$ is equivalent to the monodromy homomorphism $\mu_{p}$ being trivial.  In particular, if a marking exists, it is unique, and a family over a $1$-connected base always admits a unique marking.

Let $(\Y,B,\mu)$ be a marked family and $b_0 \in B$ a base point.
There is an associated continuous \emph{classifying map} $f : B \to \T(\Y_{b_0})$ defined by the condition that the pullback of the complex structure of $\Y_b$ to $\Y_{b_0}$ by any diffeomorphism in the preferred isotopy class represents $f(b)$.
In particular such a classifying map is obtained for any smoothly locally trivial family $\Y \to B$ over a $1$-connected base (by the existence and uniqueness of marking noted above).  The following result can be found in \cite{CAT11}.

\begin{prop}
\label{prop:local-surjectivity}
Suppose $(\Y, B, p)$ is a universal complex analytic family such that $\Y$ and $B$ are smooth complex manifolds.  Let $U\subset B$ be a $1$-connected open subset and $b\in U.$ Then, the classifying map $f : U \to \T(\Y_{b})$ is a locally surjective open map.  In particular, $f$ is a homeomorphism onto its image if and only if $f$ is injective.
\noproof
\end{prop}
By \autoref{prop:local-surjectivity}, if $f$ is injective, then the open subset $f(U)\subset \T(\Y_{b})$ canonically admits the structure of a complex manifold with a (global) chart provided by $f.$

\section{Small limit sets and Teichm\"uller space}
\label{sec:small-teich}

We now apply the generalities of \autoref{sec:teich} to the setting of Anosov families.

\begin{thm}
\label{thm:uniformization}
Let $I$ be a balanced ideal of type $(P_{A}, \F)$ and let $\rho \in \A^\star_I$.
Suppose that $V\subset \A^{\star}_I$ is an open
$G$-invariant set containing $\rho$ such that $(\eta, I)$ is $4$-small for
all $\eta \in V.$  Let $\underline{V}\subset \uA^{\star}_I$ be the corresponding quotient.  

Then, if $[\rho] \in \underline{U}\subset \underline{V}$ is any $1$-connected open subset, there exists a continuous map
\begin{align}
f: \underline{U}\rightarrow \T(\W_{\rho}^{I})
\end{align}
which is open, locally surjective, and a homeomorphism onto its image.
\end{thm}
\textbf{Remark:} When $G=\PSL(2, \C)$ and $\Gamma$ is the fundamental group of a closed, oriented surface $S$ of genus at least two, then $\uA^{\star}$ is equal to the quasi-Fuchsian space $\underline{\mathcal{QF}}_{S}.$ Moreover, there is a unique balanced ideal $I$ and $\W_{\rho}^{I}$ is smoothly diffeomorphic to $S\sqcup \overline{S}$ for any quasi-Fuchsian homomorphism $\rho: \Gamma\rightarrow \PSL(2, \C).$ The Bers Simultaneous Uniformization Theorem \cite{BER60} states that there is a biholomorphism $\underline{\mathcal{QF}}_{S}\simeq \mathcal{T}(S)\times \mathcal{T}(\overline{S}).$ 
In this vein, \autoref{thm:uniformization} should be viewed as a local uniformization theorem for Anosov homomorphisms.  

\begin{proof}[Proof of \autoref{thm:uniformization}]
Let $\underline{\W}^{I}$ denote the Anosov family over $\underline{U}.$  
Since $\underline{U}$ is $1$-connected, the family $\underline{\W}^{I}$ over $\underline{U}$ has a canonical marking and we obtain a continuous map
\begin{align}
f: \underline{U}\rightarrow \T(\W_{\rho}^{I}).
\end{align}
By \autoref{cor:universality}, the family $\underline{\W}^{I}$ over $\underline{U}$ is universal.  Therefore, by \autoref{prop:local-surjectivity}, the map $f: \underline{U}\rightarrow \T(\W_{\rho}^{I})$ is continuous, open and locally surjective.

To complete the proof, we need to show that $f$ is injective.
This will be achieved using the following result from complex-analytic extension theory:
\begin{prop}
\label{prop:flag-extension}
Suppose $h : (\F - E) \to \F$ is a locally biholomorphic map, where $E \subset \F$ has $\mathcal{H}_{2N-2}(E)=0$ for $N = \dim_\C \F$.
Then $h$ extends to a biholomorphism $\F \to \F$, and hence is the restriction of a unique element of $\textnormal{Aut}(\F).$  
\end{prop}
Before giving a proof of the proposition, we derive injectivity of $f$ from it.
Suppose that $f([\rho]) = f([\rho'])$ for $[\rho],[\rho'] \in \underline{U}$.
Then the manifolds $\W_\rho^{I}$ and $\W_{\rho'}^{I}$ are biholomorphic, and a biholomorphism between them lifts to their universal covers to give $h : \Omega_{\rho}^{I} \to \Omega_{\rho'}^{I}$.
Since $(\rho, I)$ is $4$-small, this map satisfies the hypotheses of \autoref{prop:flag-extension}, and therefore $h$ is the restriction of an automorphism of $\F.$  Moreover, by the definition of Teichm\"{u}ller space, this automorphism is smoothly isotopic to the identity, and therefore $h$ is the restriction of an element $g\in G=\textnormal{Aut}_{0}(\F).$  
It follows that $\rho=g\cdot \rho^{\prime} \cdot g^{-1}.$
Thus $[\rho] = [\rho']$, as desired.

\begin{proof}[Proof of \autoref{prop:flag-extension}]
The following argument is inspired by the results in \cite{McKay2009}.

By a theorem of Chirka \cite{Chirka1996}, the holomorphic map $h: (\mathcal{F} - E) \to \mathcal{F}$ extends to a strongly meromorphic mapping $\mathcal{F} \to \mathcal{F}$.
Such a meromorphic map in particular gives a holomorphic extension to much of $\mathcal{F}$; precisely, there exists a holomorphic extension $\tilde{h} : (\mathcal{F} - B) \to \mathcal{F}$ where $B \subset E$ is a closed analytic set.
By the Hausdorff measure condition, $B$ has complex codimension at least two in $\F$.

We claim that this holomorphic extension $\tilde{h}$ is also a local biholomorphism, which of course is already known except at points of $E-B$.
Here we use use an argument inspired by \cite[Lemma 3.18]{McKay2009}:
Work in local coordinates about a neighborhood $U$ of a point in $E - B$, chosen so that $U \cap B = \emptyset$.
Let $\delta$ denote the Jacobian determinant of $d\tilde{h}$ relative to this coordinate system, which is a nonvanishing holomorphic function on $U - E$.
Then $1/\delta$ is holomorphic on $U - E$ and thus extends holomorphically to $U$ by Shiffman's theorem \cite[Lemma 3]{Shi68}.
Thus $\delta$ is nonzero throughout $U$, and $\tilde{h}$ is a local biholmorphism.

Finally, by a Lemma of Ivashkovich (proved e.g.~in \cite[Lemma 3.45]{McKay2009}), the homogeneity of $\mathcal{F}$ implies that the meromorphic local biholomorphism $\tilde{h}$ in fact extends to a local biholomorphism $\tilde{h}: \F \rightarrow \F.$
Since $\mathcal{F}$ is compact and simply connected, $\tilde{h}$ is an automorphism of $\F.$
\end{proof}

As noted earlier, this also completes the proof of \autoref{thm:uniformization}.
\end{proof}

As in the previous section, we can adapt this theorem to the case of surface groups; the following is a more detailed version of \autoref{thm:main-local} from the introduction:
\begin{cor}\label{cor: simun}
Let $G$ be a complex simple group not of type $A_{1}, A_{2}, A_{3}$ or $B_{2}$ and $I$ any balanced ideal of type $(P_{A}, \F).$  Let $S$ be a closed orientable surface and $\Gamma=\pi_{1}(S).$  Then there exists a connected open non-empty $G$-invariant set $\mathcal{U}\subset \A_{I}^{\star}$ such that
\begin{rmenumerate}
\item  The set $\mathcal{U}$ contains all $G$-quasi-Fuchsian homomorphisms.  If $G$ is not of type $F_{4}, E_{6}, E_{7}$ or $E_{8},$ then $\mathcal{U}$ also contains the set of $G$-Hitchin homomorphisms.
\item The corresponding quotient $\underline{\mathcal{U}}\subset \uA_{I}^{\star}$ is $1$-connected.  
\item For any $\rho\in \mathcal{U},$ the classifying map $f: \underline{\mathcal{U}}\rightarrow \T(\W_{\rho}^{I})$ is open, locally surjective, and a homeomorphism onto its image.
\item There is a commutative diagram
\begin{center}
\begin{tikzcd}
\underline{\mathcal{QF}}_{S} \arrow{r}{\simeq} \arrow{d}
&\T(S) \times \T(\overline{S}) \arrow{d} \\
\underline{\mathcal{U}} \arrow{r}{f}
& \T(\W_{\rho}^{I}).
\end{tikzcd}
\end{center}
where the top horizontal arrow is the simultaneous uniformization map.
\end{rmenumerate}
\end{cor}
\begin{proof}
The existence of $\mathcal{U}$ is guaranteed by \autoref{cor:surface-completeness}.  By \autoref{prop:GQF} and \autoref{prop:GHIT}, the spaces of conjugacy classes of $G$-quasi-Fuchsian and $G$-Hitchin homomorphisms are $1$-connected (even contractible), so upon shrinking $\mathcal{U}$ we may assume $\underline{\mathcal{U}}$ is $1$-connected.  Then, \autoref{cor:surface-completeness} in conjunction with \autoref{thm:uniformization} implies $(3).$   

Finally, since $\underline{\mathcal{QF}}_{S}\simeq \T(S) \times \T(\overline{S}) $, the right vertical arrow in $(4)$ is uniquely defined so that the diagram commutes.  
\end{proof}
The above result shows that the simultaneous uniformization map admits an analytic continuation to the setting of homomorphisms $\Gamma\rightarrow G.$
We leave unanswered the important problem of understanding the maximal domain of definition of this analytic continuation.

\section{Uniform lattices in \texorpdfstring{$\textnormal{SO}_{0}(n,1)$}{SO(n,1)}}
\label{sec:son1}

We close the paper with a brief discussion of another class of interesting examples.  Let $\Gamma<\textnormal{SO}_{0}(n,1)$ be a torsion free cocompact lattice considered as a subgroup of $\PSL(n+1, \C)$ via the natural inclusion.  Let $\F_{1,n}$ be the flag variety consisting of flags $\ell\subset H\subset \C^{n+1}$ where $\ell$ is a line and $H$ is a hyperplane.  Then, $\PSL(n+1, \C)$ acts transitively on $\F_{1,n}$ and there is a parabolic subgroup $P_{1,n}<\PSL(n+1, \C)$ such that $\PSL(n+1, \C)/P_{1,n}\simeq \F_{1,n}.$ 

It is straightforward to verify that $\rho: \Gamma\rightarrow \PSL(n+1, \C)$ is $P_{1,n}$-Anosov.  Indeed, the limit curve $\xi_{\rho}: \partial\Gamma\rightarrow \F_{1,n}$ is obtained as follows.  The group $\Gamma$ acts on the hyperbolic space $\mathbb{H}^{n}$ and the limit set (in the sense of accumulation set of an orbit) gives a homeomorphism
$\partial \Gamma\simeq \partial \mathbb{H}^{n}\simeq S^{n-1}.$  

Hyperbolic space can be viewed in its projective model $\mathbb{H}^{n}\subset \mathbb{RP}^{n},$ and taking tangent hyperplanes to the boundary $(n-1)$-sphere gives a map $\partial \Gamma\rightarrow \F_{1,n-1}(\R)\subset \F_{1,n-1}$ where $\F_{1,n-1}(\R)$ is the real points of the complex projective variety $\F_{1,n-1}.$ This defines the limit curve $\xi_{\rho}: \partial \Gamma\rightarrow \F_{1,n-1}.$  The quotient of $\mathbb{H}^{n}\subset \mathbb{RP}^{n}$ is an example of a strictly convex real projective manifold.

Let $\F$ be any other flag variety of $\PSL(n+1, \C)$ and $I$ be a balanced ideal of type $(P_{1,n}, \F).$  We note that if $\F$ is the variety of complete flags, there always exists such an ideal (see \cite{DS20} for a discussion of ideals in this setting).  Then, we obtain a  domain $\Omega_{\rho}^{I}\subset \F$ and, when this domain is non-empty, a compact quotient manifold $\W_{\rho}^{I}.$  

By \autoref{thm:pent} we have a diagram:
\begin{equation}
\begin{tikzcd}
&  Z^1(\Gamma,\sl(n+1,\C)_\rho) \ar[rd,"\KS_\rho"] \ar[d] \\
0 \ar[r] &  H^1(\Gamma,H^0(\Omega_{\rho}^{I}, \Theta_{\Omega_\rho^{I}})) \ar[r] & H^1(\W_{\rho}^{I},\Theta_{\W_\rho^{I}}) \ar[r] & H^0(\Gamma, H^1(\Omega_{\rho}^{I}, \Theta_{\Omega_\rho}^{I})).
\end{tikzcd}
\end{equation}

If $n\geq 4,$ there always exists some flag variety $\F$ and a balanced ideal $I$ of type $(P_{1,n-1}, \F)$ such that $(\rho, I)$ is $4$-small, and for such a pair we obtain an isomorphism
$H^{1}(\Gamma, \sl(n+1,\C)_\rho)\simeq H^{1}(\W_{\rho}^{I}, \Theta_{\W_{\rho}^{I}}).$
However, the question of whether $H^{1}(\Gamma, \sl(n+1,\C)_\rho)$ is non-trivial is very delicate.  If the closed hyperbolic manifold $\Gamma\backslash \mathbb{H}^{n}$ contains a totally geodesic hypersurface, then there exist non-trivial deformations given by a bending procedure, and hence the corresponding complex manifold $\W_{\rho}^{I}$ will have non-trivial complex deformations.  

Meanwhile, in the absence of a totally geodesic hypersurface, there are many examples where $H^{1}(\Gamma, \sl(n+1,\C)_\rho)=\{0\};$ this has been analyzed most thoroughly for $n=3,$ but no recognizable pattern has emerged (see \cite{CLM18} for a detailed discussion).
 We remark that this study essentially reduces to the study of deformations of the strictly convex real projective manifold given by the $\Gamma$-action on $\mathbb{H}^{n}\subset \mathbb{RP}^{n}.$  

Thus we obtain a collection of complex manifolds labeled by the data $(\Gamma, I)$, some of which are infinitesimally rigid, and some of which admit nontrivial first-order deformations of complex structure.

\appendix
\section{Appendix}
\label{app:homological}

The homological algebra appearing in \autoref{sec:ks} is entirely classical, but among geometric topologists interested in Anosov homomorphisms, it is probably less well-known.
This appendix provides additional detail on one aspect: the case of the Grothendieck spectral sequence used in the proof of \autoref{thm:pent}.
The results here are particular examples of those in the original paper of Grothendieck \cite{GRO57}.

Let $p: X\rightarrow Y$ be a regular covering of complex manifolds with deck group $\Gamma$ and $\mathcal{L}$ a $\Gamma$-equivariant locally free sheaf on $X.$  There is a canonical locally free sheaf $\mathcal{L}^{\Gamma}$ on $Y$ such that $p^{\star}\mathcal{L}^{\Gamma}\simeq \mathcal{L}$; we will call this the \emph{descent} of $\mathcal{L}$ to $Y$.  

Next, the Dolbeault complex  
is the complex
\begin{align}
A^{0}(X, \mathcal{L})\xrightarrow{\overline{\partial}} A^{0,1}(X, \mathcal{L})\rightarrow ...
\end{align}
of smooth $(0,q)$-forms on $X$ with values in the associated holomorphic vector bundle.  Since $\Gamma$ acts holomorphically on $X,$ the $\C$-vector spaces $A^{0,q}(X, \mathcal{L})$ are $\Gamma$-modules via pullback of forms, and thus 
\begin{align}
A^{0}(X, \mathcal{L})\xrightarrow{\overline{\partial}} A^{0,1}(X, \mathcal{L})\xrightarrow{\overline{\partial}} ...
\end{align}
is an object of the abelian category of complexes of $\Gamma$-modules.  Taking $\Gamma$-invariants defines a functor from this category to the category of complexes of $\C$-vector spaces which is left exact, and the right derived functors define the (hyper)-cohomology groups $\mathbb{H}^{i}(\Gamma, A^{0, \param}(X, \mathcal{L})).$  
\begin{thm}
Let $\mathcal{L}^{\Gamma}$ be the descent of $\mathcal{L}$ to $Y.$ Then, there are a canonical isomorphisms $\mathbb{H}^{k}(\Gamma, A^{0, \param}(X, \mathcal{L}))\simeq H^{k}(Y, \mathcal{L}^{\Gamma})$ for every $k\geq 0.$
\end{thm}
Simultaneously taking the complex of inhomogeneous cochains leads to a double complex $E_{0}^{p,q}:=C^{p}(\Gamma, A^{0,q}(X, \mathcal{L}))$ whose horizontal differentials are given by the  differential on inhomogeneous group cochains 
$d_{\Gamma}: C^{p}(\Gamma, A^{0,q}(X, \mathcal{L}))\rightarrow C^{p+1}(\Gamma, A^{0,q}(X, \mathcal{L}))$ and vertical differentials are given by $\overline{\partial}: C^{p}(\Gamma, A^{0,q}(X, \mathcal{L}))\rightarrow C^{p}(\Gamma, A^{0,q+1}(X, \mathcal{L})).$  

\begin{thm}
\label{thm:applied-groth}
The degree $k$ cohomology of the total complex of $E_{0}^{p,q}:=C^{p}(\Gamma, A^{0,q}(X, \mathcal{L}))$ is canonically isomorphic to 
$\mathbb{H}^{k}(\Gamma, A^{0, \param}(X, \mathcal{L})).$ 

Therefore, there is a spectral sequence with $E_{2}^{p,q}:=H^{p}(\Gamma, H^{q}(X, \mathcal{L}))$ which converges to
$$\mathbb{H}^{p+q}(\Gamma, A^{0, \param}(X, \mathcal{L}))\simeq H^{p+q}(Y, \mathcal{L}^{\Gamma}).$$   
\end{thm}
For our applications, the regular covering is $\Omega_{\rho}^{I}\rightarrow \W_{\rho}^{I},$ the $\Gamma$-equivariant sheaf is $\Theta_{\Omega_{\rho}^{I}},$ and its descent is $\Theta_{\Omega_{\rho}^{I}}^{\Gamma}=\Theta_{\W_{\rho}^{I}}.$  Hence, we obtain the spectral sequence with $E_2$-page $H^{p}(\Gamma, H^{q}(\Omega_{\rho}^{I}, \Theta_{\Omega_{\rho}^{I}}))$ converging to $H^{p+q}(\W_{\rho}^{I}, \Theta_{\W_{\rho}^{I}})$ that is used in \autoref{sec:ks}.

\bibliographystyle{habbrv}
\bibliography{ds2}

\vspace{1.5em}

\noindent Department of Mathematics, Statistics, and Computer Science\\
University of Illinois at Chicago\\
\texttt{david@dumas.io}\\

\medskip

\noindent Mathematisches Institut\\
Ruprecht-Karls-Universit\"{a}t Heidelberg\\
\texttt{asanders@mathi.uni-heidelberg.de}\\

\end{document}